\documentclass[12pt]{article}
\usepackage[T1]{fontenc}
\usepackage[utf8]{inputenc}

\usepackage[american]{babel}
\usepackage[babel, final]{microtype}
\usepackage[all]{xy}
\usepackage{ifpdf}
\usepackage{comment}
\usepackage{multirow}
\usepackage{soul}

\usepackage[pdftex, dvipsnames]{xcolor}
\usepackage[pdftex, final]{graphicx}
\usepackage{pinlabel}
\usepackage[pdftex,%
    nomarginpar,%
    lmargin=1in,%
    rmargin=1in,%
    tmargin=1in,%
    bmargin=1in,%
]{geometry}
\usepackage{amsmath,amsthm,amssymb,amsfonts,mathrsfs,tikz,tikz-cd,bm,verbatim,mathtools,hyperref,caption,enumitem,csquotes,float}
\usepackage{thmtools, thm-restate, makecell}
\allowdisplaybreaks

\usepackage[normalem]{ulem}
\usepackage[margin=1cm]{caption}

\usepackage[hang,flushmargin]{footmisc}

\hypersetup{
    colorlinks=true,
    linkcolor=NavyBlue,
    citecolor=NavyBlue
}

\usepackage[backend=biber,style=alphabetic,sorting=nyt, url=false,doi=false,isbn=false,backref=true,maxnames=10]{biblatex}
\DefineBibliographyStrings{english}{
  backrefpage={$\uparrow$},
  backrefpages={$\uparrow$}
}

\setlist{noitemsep}

\addto\extrasamerican{%
}

\newtheorem{corollary}{Corollary}[section]
\newtheorem{proposition}{Proposition}[section]
\newtheorem{lemma}{Lemma}[section]

\newtheorem{conjecture}{Conjecture}[section]

\theoremstyle{definition}
\newtheorem{definition}{Definition}[section]
\newtheorem{example}{Example}[section]
\newtheorem{remark}{Remark}[section]

\newtheorem{question}{Question}[section]
\newtheorem{problem}{Problem}[section]

\makeatletter
\let\c@conjecture=\c@theorem
\let\c@corollary=\c@theorem
\let\c@proposition=\c@theorem
\let\c@lemma=\c@theorem
\let\c@definition=\c@theorem
\let\c@example=\c@theorem
\let\c@remark=\c@theorem
\let\c@equation\c@theorem
\let\c@question\c@theorem
\let\c@problem\c@theorem
\let\c@fact\c@theorem
\makeatother

\def\makeautorefname#1#2{\expandafter\def\csname#1autorefname\endcsname{#2}}
\makeautorefname{proposition}{Proposition}
\makeautorefname{lemma}{Lemma}
\makeautorefname{definition}{Definition}
\makeautorefname{example}{Example}
\makeautorefname{question}{Question}
\makeautorefname{conjecture}{Conjecture}
\makeautorefname{section}{Section}
\makeautorefname{claim}{Claim}
\makeautorefname{equation}{Equation}
\makeautorefname{remark}{Remark}
\makeautorefname{corollary}{Corollary}
\makeautorefname{fact}{Fact}
\makeautorefname{problem}{Problem}

\DeclareMathOperator{\id}{id}

\DeclareMathOperator{\lk}{lk}

\makeatletter
\def\@tocline#1#2#3#4#5#6#7{\relax
  \ifnum #1>\c@tocdepth % then omit
  \else
    \par \addpenalty\@secpenalty\addvspace{#2}%
    \begingroup \hyphenpenalty\@M
    \@ifempty{#4}{%
      \@tempdima\csname r@tocindent\number#1\endcsname\relax
    }{%
      \@tempdima#4\relax
    }%
    \parindent\z@ \leftskip#3\relax \advance\leftskip\@tempdima\relax
    \rightskip\@pnumwidth plus4em \parfillskip-\@pnumwidth
    #5\leavevmode\hskip-\@tempdima
      \ifcase #1
       \or\or \hskip 1em \or \hskip 2em \else \hskip 3em \fi%
      #6\nobreak\relax
    \dotfill\hbox to\@pnumwidth{\@tocpagenum{#7}}\par
    \nobreak
    \endgroup
  \fi}
\makeatother

\makeatletter
\def\@tempa#1{\@xp\@tempb\meaning#1\@nil#1}
\def\@tempb#1>#2#3 #4\@nil#5{%
  \@xp\ifx\csname#3\endcsname\mathaccent
    \@tempc#4?"7777\@nil#5%
  \else
    \PackageWarningNoLine{amsmath}{%
      Unable to redefine math accent \string#5}%
  \fi
}
\def\@tempc#1"#2#3#4#5#6\@nil#7{%
  \chardef\@tempd="#3\relax\set@mathaccent\@tempd{#7}{#2}{#4#5}}

\@tempa\widehat
\makeatother

\title{Classical invariants of spiral knots}

\author{Sarah Blackwell, Ashish Das, Sydney Mayer,\\Luke Moyar, Faisal Leo Quraishi, and Ryan Stees}

\date{}

\begin{document}
\maketitle
\begin{abstract}
Torus knots are an important family of knots about which much is understood; invariants of torus knots often exhibit nice formulas, making them convenient and fundamental building blocks for examples in knot theory. Spiral knots, defined and first studied in \cite{BETVWWY}, are a braid-theoretic generalization of torus knots, but comparatively not much is known about this broader family of knots. We give a general recursive formula for the Alexander polynomials of spiral knots, and from this we derive several properties of spiral knots, including a simple genus formula. Additionally, we investigate the consequences these results have on classification questions.
\end{abstract}

%-------------------------------------------------%
\section{Introduction} \label{sec:intro}

Torus knots are among the most thoroughly-studied families of knots in $S^3$. Often, their invariants admit simple formulas, making them convenient building blocks for examples in knot theory. For instance, torus knots serve as a point of reference in the study of concordance and 4-dimensional properties of knots. Kronheimer-Mrowka's proof of the Milnor conjecture on the smooth 4-genus of torus knots is one of the first such results for a family of knots \cite{KronheimerMrowka}. \emph{Slice torus invariants}, which include the $\tau$-invariant of knot Floer homology \cite{OzsvathSzabo} and the $s$-invariant of Khovanov homology \cite{Rasmussen}, are an entire family of concordance invariants defined by their behavior on torus knots \cite{Lewark}. The subgroup of the concordance group generated by torus knots has been studied extensively \cites{Litherland,HKL,CKP}, and cables of and surgeries on torus knots are heavily featured in recent work on concordance and homology cobordism \cites{AcetoGolla,PinzonCaicedo,Lokteva}.

A number of generalizations of torus knots, such as twisted torus knots \cites{DRS,ParkShahab}, have also garnered recent attention. In this paper we consider \emph{spiral knots}, a braid-theoretic generalization of torus knots first defined in \cite{BETVWWY} which additionally encompasses the family of weaving knots \cite{DiPrisaSavk}.

\begin{definition} \label{def:spiral}
Let $p$ and $q$ be relatively prime natural numbers with $p,q\geq 2$, and let $\varepsilon=(\varepsilon_i)_{i=1}^{p-1}\in\{\pm 1\}^{p-1}$. The \emph{spiral knot} $S(p,q,\varepsilon)$ is the closure of the braid \[\Bigg(\prod_{i=1}^{p-1}\sigma_i^{\varepsilon_i}\Bigg)^q\in B_p,\] where $B_p$ is the braid group on $p$ strands with standard generators $\sigma_1,\dots,\sigma_{p-1}$.
\end{definition}

The above definition makes sense for $q=1$, but every $S(p,1,\varepsilon)$ is unknotted, so we will typically assume $q\geq 2$ unless stated otherwise.
Note that when $\varepsilon=(1,1,\dots,1)$, the spiral knot $S(p,q,\varepsilon)$ is the torus knot $T_{p,q}$. Further note that any $S(p,q,\varepsilon)$ is $q$-periodic and is the preimage of $S(p,1,\varepsilon)$ under the $q$-fold branched cover of $S^3$ branched over the braid axis. The preimage of the branch set under this branched cover is the braid axis for $S(p,q,\varepsilon)$.

The subfamilies consisting of torus knots and weaving knots are well-studied. In contrast, little is known about the larger family of spiral knots as a whole. Brothers et al. study spiral knots with $q$ a prime power in \cite{BETVWWY}, and Kim, Taalman, and the last author show certain families of spiral knots have determinants which admit related recursive formulas in \cite{KST}.

In this paper, we give a general formula for the Alexander polynomials of all spiral knots.

\begin{restatable*}{theorem}{Alex} \label{Ther:Alex} The Alexander polynomial $\Delta_{S(p,q,\varepsilon)}(t)$ of the spiral knot $S(p,q,\varepsilon)$ is given by the formula
\[\Delta_{S(p,q,\varepsilon)}(t)  = \prod\limits_{\ell = 1}^{q-1} 
\mathcal{C}_{p-1}\left(e^{\frac{2\pi \ell}{q}i},t\right)\]
where $\mathcal{C}_k(x,t)$ is defined recursively for $k\geq 2$ by 
\[\mathcal{C}_k = \left(\frac{\mu(k)^2}{t}+x\right)\mathcal{C}_{k-1} - \left(\frac{\mu(k-1)\mu(k)x}{t}\right)\mathcal{C}_{k-2}\]
with $\mathcal{C}_0 = 1$, $\mathcal{C}_1 = \frac{\mu(1)^2}{t}+x$, and $\mu(i) = \begin{cases}
1, & \varepsilon_i = 1 \\
t, & \varepsilon_i = -1 \\
0, & i\notin \{1,\dots, p-1\}
\end{cases}$.
\end{restatable*}

\noindent We expect that the polynomials $\mathcal{C}_k$ are related to the multivariable Alexander polynomial of the two-component link consisting of the spiral knot $S(p,1,\varepsilon)$ and its braid axis which is the branch set of a $q$-fold branched cover producing the knot $S(p,q,\varepsilon)$. The above formula appears to be an explicit version of the general form of the Alexander polynomial of a periodic knot given by Murasugi in \cite{Murasugi}. (See \autoref{sec:future} for further discussion.)

From this formula, we derive several additional properties of spiral knots, including a formula for the 3-genus of a spiral knot which agrees with the formula for torus knots.

\begin{restatable*}{theorem}{genus} \label{Ther:genus}
The genus of the spiral knot 
$S(p,q,\varepsilon)$ is 
$$g(S(p,q,\varepsilon)) = \frac{(p-1)(q-1)}{2}.$$    
\end{restatable*}

\noindent It is worth noting that the Milnor conjecture, proved first by Kronheimer-Mrowka in \cite{KronheimerMrowka} and later in \cites{OzsvathSzabo,Rasmussen}, which states that this formula also holds for the smooth 4-genus of torus knots, does not hold for spiral knots, as the knot $10_{123}=S(3,5,(1,-1))$ is slice. In future work, we hope to study general concordance properties of spiral knots.

We additionally consider whether each non-torus spiral knot corresponds to a unique triple $(p,q,\varepsilon)$. Recall that for torus knots we have $T_{p,q}=T_{q,p}$. Along these lines, we prove:

\begin{restatable*}{theorem}{swappq} \label{Ther:swappq}
$S(p,q,\varepsilon) = S(q,p,\varepsilon')$ for some $\varepsilon'$ if and only if $S(p,q,\varepsilon)$ is a torus knot.
\end{restatable*}

\noindent In this way, the subfamily of torus knots may be thought of as a ``diagonal'' within the family of spiral knots.

% \subsection*{Outline}

% After providing some background in \autoref{sec:background}, we prove \autoref{Ther:Alex} regarding the Alexander polynomials of spiral knots in \autoref{sec:alexander}. We then consider consequences of the Alexander polynomial formula in \autoref{sec:classification} and discuss some examples. \autoref{sec:future} discusses future directions, including a study of bridge number and meridional rank.

\subsection*{Acknowledgments}
Most of the work for this project took place at the  UVA Topology REU in the summer of 2024. The authors would like to thank the University of Virginia for hosting them, Thomas Koberda and Slava Krushkal for organizing the REU, and Louisa Liles and Yangxiao Luo for their mentorship and helpful conversations. This REU program was supported by the NSF Research Training Group grant DMS-1839968. SB was additionally supported by the NSF Postdoctoral Research Fellowship DMS-2303143. We also acknowledge the KnotInfo table of knots \cite{KnotInfo} and KnotFolio \cite{KnotFolio}, which were incredibly useful tools throughout the project, and thank Benjamin Bode, Chuck Livingston, and Kyle Miller for helpful correspondences.

%%%%%%%%%%%%%%%%%%%%%%%%%%%%%%%%
\setcounter{tocdepth}{2}
\tableofcontents
%%%%%%%%%%%%%%%%%%%%%%%%%%%%%%%%

%-------------------------------------------------%
\section{Background} \label{sec:background}

\subsection{Torus and spiral knots}
A \emph{knot} is a smooth embedding $S^1\hookrightarrow S^3$ considered up to ambient isotopy. 
\textit{Torus knots} are the family of knots $\{T_{p,q}\}$ that can sit on the surface of an unknotted torus in $S^3$. 
They are parametrized by coprime pairs of integers $(p,q)$, where $p$ and $q$ correspond to the number of times the knot wraps around the torus in each $S^1$ direction. (If $p,q$ are not coprime, then $T_{p,q}$ is a link with $\gcd(p,q)$ components.)
Braids give a particularly convenient way of representing torus knots. 
The torus knot $T_{p,q}$ is the closure of the braid $(\sigma_1\sigma_2\cdots\sigma_{p-1})^q$, where $\sigma_i$ is the standard generator of the braid group $B_p$ on $p$ strands represented by strand $i$ crossing over strand $i+1$. See \autoref{fig:torus}. 

\begin{figure}[h]
\centering
    \begin{minipage}[c]{0.3\linewidth}   
    \includegraphics[width=\linewidth]{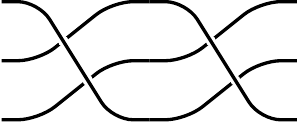}
    \end{minipage}
    \hspace{2em}
    \begin{minipage}[c]{0.4\linewidth}
    \includegraphics[width=\linewidth]{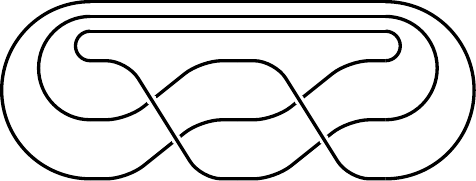}
    \end{minipage}
    \caption{The braid $(\sigma_1\sigma_2)^2$ and its closure, the torus knot $T_{3,2}$.
    \label{fig:torus}}
\end{figure}

Spiral knots (recall \autoref{def:spiral}) are a braid-theoretic generalization of torus knots that were first defined in \cite{BETVWWY}. 
The spiral knot $S(p,q,\varepsilon)$ is the closure of the braid $(\sigma_1^{\varepsilon_1}\sigma_2^{\varepsilon_2}\cdots\sigma_{p-1}^{\varepsilon_{p-1}})^q$, where $\varepsilon$ is the vector $\varepsilon=(\varepsilon_1,\varepsilon_2,\dots,\varepsilon_{p-1})\in\{\pm 1\}^{p-1}$; see \autoref{fig:spiral}. 
Compared with the previous braids whose closures are torus knots, we now allow each $\sigma_i$ to appear with an exponent of $+1$ or $-1$. 
Using this notation, the torus knot $T_{p,q}$ is the spiral knot $S(p,q,\varepsilon)$ with $\varepsilon=(1,1,\dots,1)$. 
By construction the spiral knot $S(p,q,\varepsilon)$ is $q$-periodic.

\begin{figure}[h]
\centering
    \begin{minipage}[c]{0.4\linewidth}   
    \includegraphics[width=\linewidth]{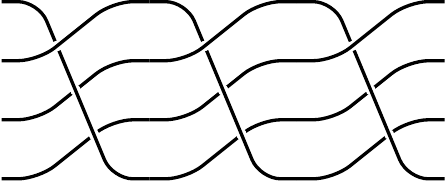}
    \end{minipage}
    \hspace{2em}
    \begin{minipage}[c]{0.5\linewidth}
    \includegraphics[width=\linewidth]{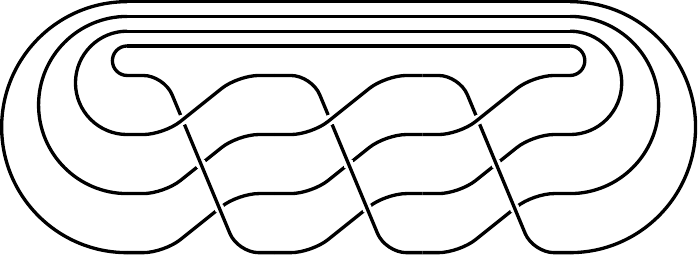}
    \end{minipage}
    \caption{The braid $(\sigma_1^{-1}\sigma_2\sigma_3)^3$ and its closure, the spiral knot $S(4,3,(-1,1,1))$.
    \label{fig:spiral}}
\end{figure}

In this paper, we will usually adopt the convention that $\varepsilon_1 = 1$. Note that given $S(p,q,\varepsilon)$, the knot $S(p,q,-\varepsilon)$ is its mirror image, but we will generally not distinguish between knots and their mirrors in this work. Similarly, we only consider $\varepsilon$ vectors up to inverting the order of the coordinates, as this does not change the knot type; for instance, we consider $(1,-1,1,1)$ and $(1,1,-1,1)$ to be the same $\varepsilon$ vector.

\subsection{Invariants of torus and spiral knots}
Torus knots are well-understood in the sense that many of their invariants have been computed and admit simple formulas.
Recall the following invariants of knots $K\subset S^3$:
\begin{itemize}
\item The \emph{genus} of $K$ is the minimum genus $g(K)$ of an orientable (Seifert) surface in $S^3$ which $K$ bounds.
\item The \emph{Alexander polynomial} of $K$ is $\Delta_K(t)=\det(M-tM^T)$, where $M$ is a Seifert matrix corresponding to any Seifert surface for $K$.
\item The \emph{determinant} of $K$ is $\det(K)=\det(M+M^T)$, where $M$ is again a Seifert matrix for $K$. Note that $\det(K)=\Delta_K(-1)$.
\end{itemize}
For the torus knot $T_{p,q}$, we have genus
\[
g(T_{p,q}) = \frac{(p-1)(q-1)}{2}
\]
and Alexander polynomial
\[
\Delta_{T_{p,q}}(t) = \frac{(t^{pq}-1)(t-1)}{(t^p-1)(t^q-1)}.
\]
As a consequence, Breiland-Oesper-Taalman \cite{BOT} note that the determinant of $T_{p,q}$ is
\[
\det(T_{p,q}) = \begin{cases}
    1, & \text{$p$ and $q$ are both odd}\\
    p, & \text{$p$ is odd and $q$ is even}\\
    q, & \text{$q$ is odd and $p$ is even}
\end{cases}.
\]

It is shown in \cite{BETVWWY} that the above genus formula holds more generally for spiral knots $S(p,q,\varepsilon)$ when $q$ is a prime power.
In \autoref{Ther:genus} we show that this formula holds for \emph{all} spiral knots.
In contrast, Kim, Taalman, and the last author demonstrate that determinants of spiral knots exhibit very different patterns than the simple formula for torus knots \cite{KST}. 
Indeed, we show in \autoref{Ther:Alex} that the Alexander polynomials of spiral knots exhibit a more complicated formula than the one above for torus knots.

%-------------------------------------------------%
\section{The Alexander polynomials of spiral knots} \label{sec:alexander}

For the sake of notation, throughout this section and the next we will often write $\Delta(t)$ to mean $\Delta_{S(p,q,\varepsilon)}(t)$.

\subsection{Calculating the Seifert matrix}
Let $S(p,q,\varepsilon)$ be a spiral knot. Recall that the Alexander polynomial for any knot can be computed from a Seifert matrix of that knot. To find a form for this matrix given $p$, $q$, and $\varepsilon$, we first obtain a Seifert surface for the spiral knot $S(p,q,\varepsilon)$ using Seifert's algorithm. We refer to the surface given by Seifert's algorithm on the braid form of a spiral knot $S(p,q,\varepsilon)$ as the \textit{cake surface} for $S(p,q,\varepsilon)$, because of the resemblance it bears to a tiered cake; see \autoref{fig:cake}.

\begin{figure}[h]
    \begin{minipage}[c]{0.25\linewidth}   
    \includegraphics[width=\linewidth, ]{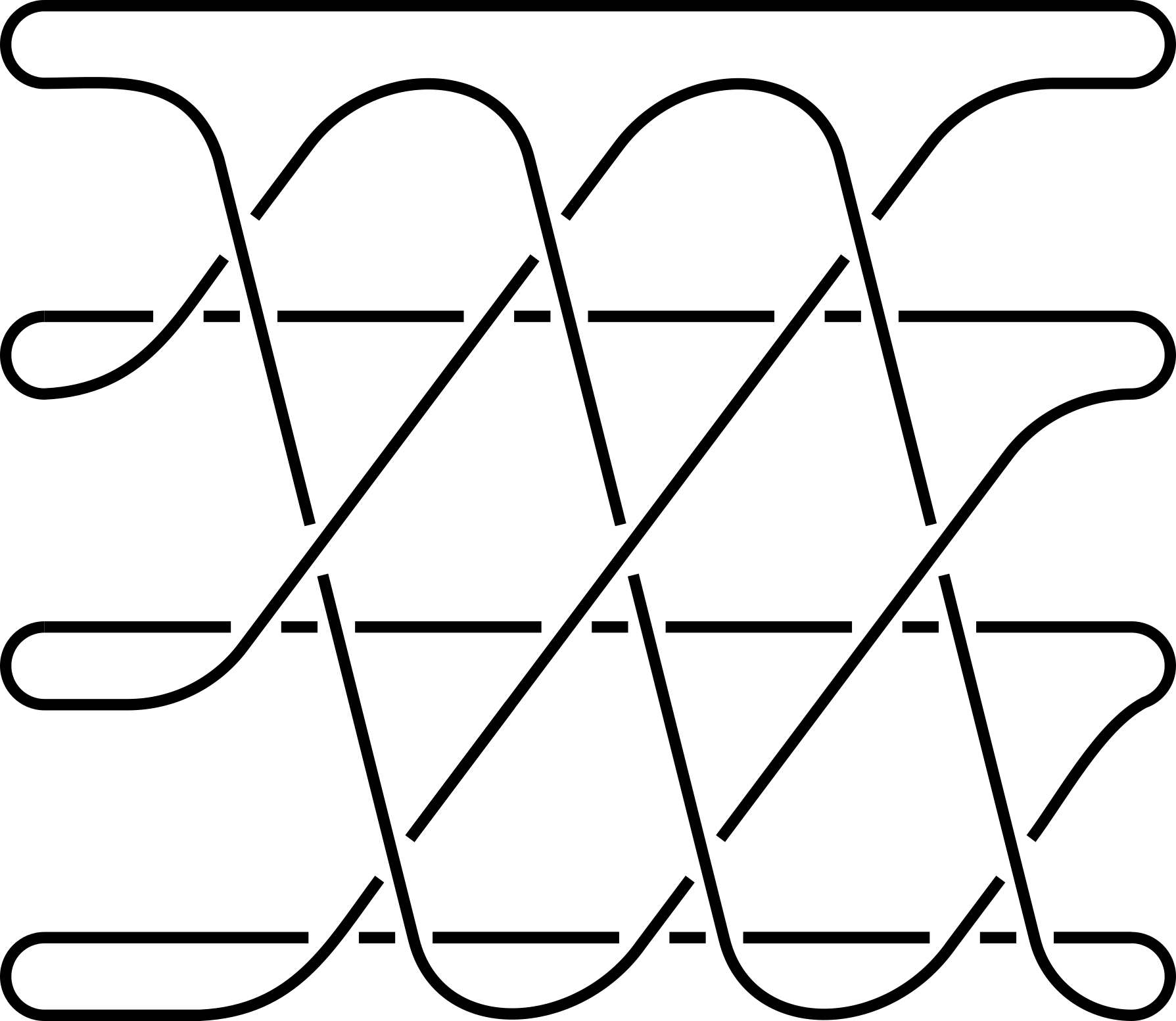}
    \end{minipage}
    \hfill
    \begin{minipage}[c]{0.4\linewidth}
    \includegraphics[width=\linewidth, ]{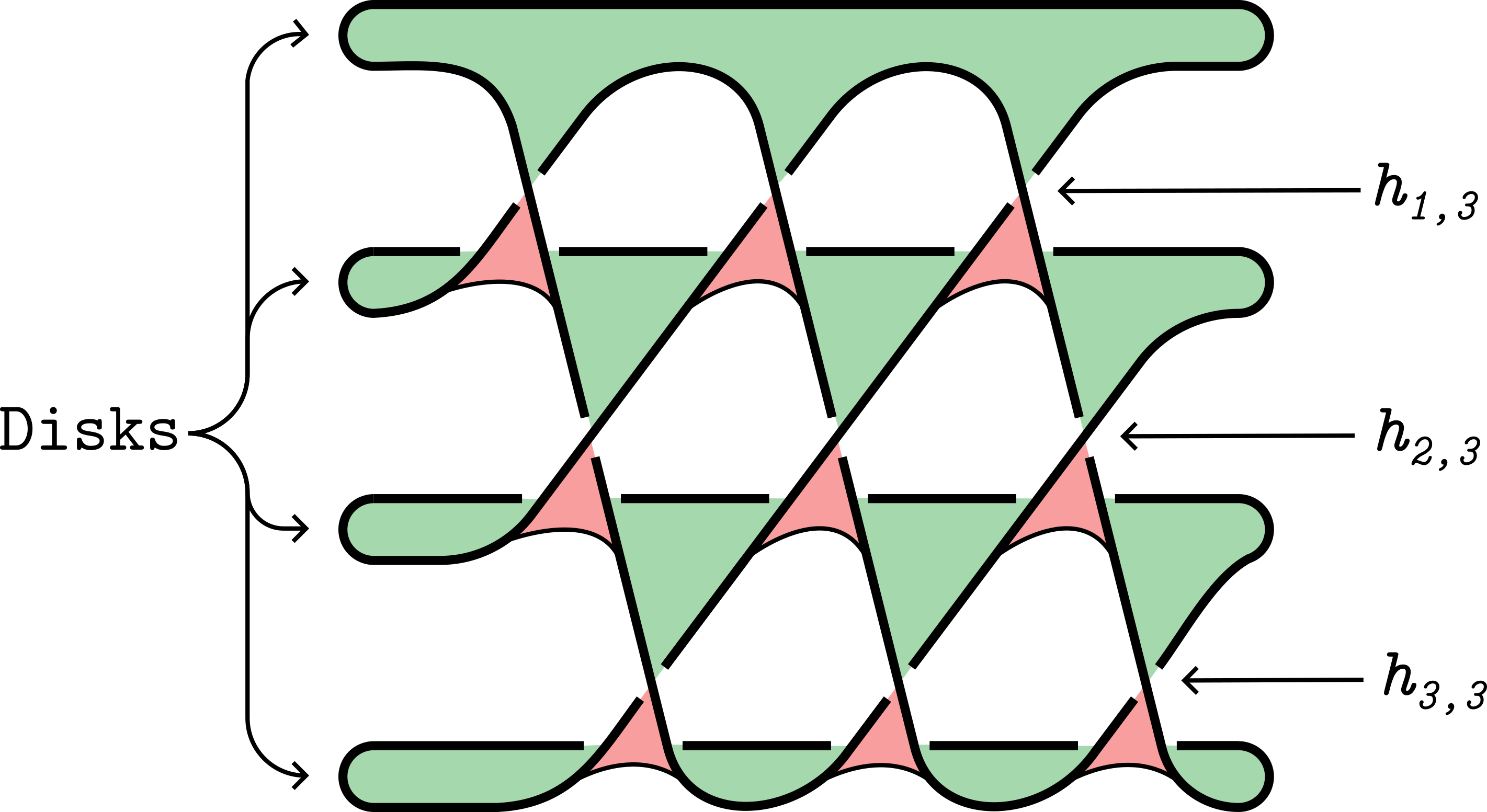}
    \end{minipage}
    \hfill
    \begin{minipage}[c]{0.3\linewidth}
    \includegraphics[width=\linewidth, ]{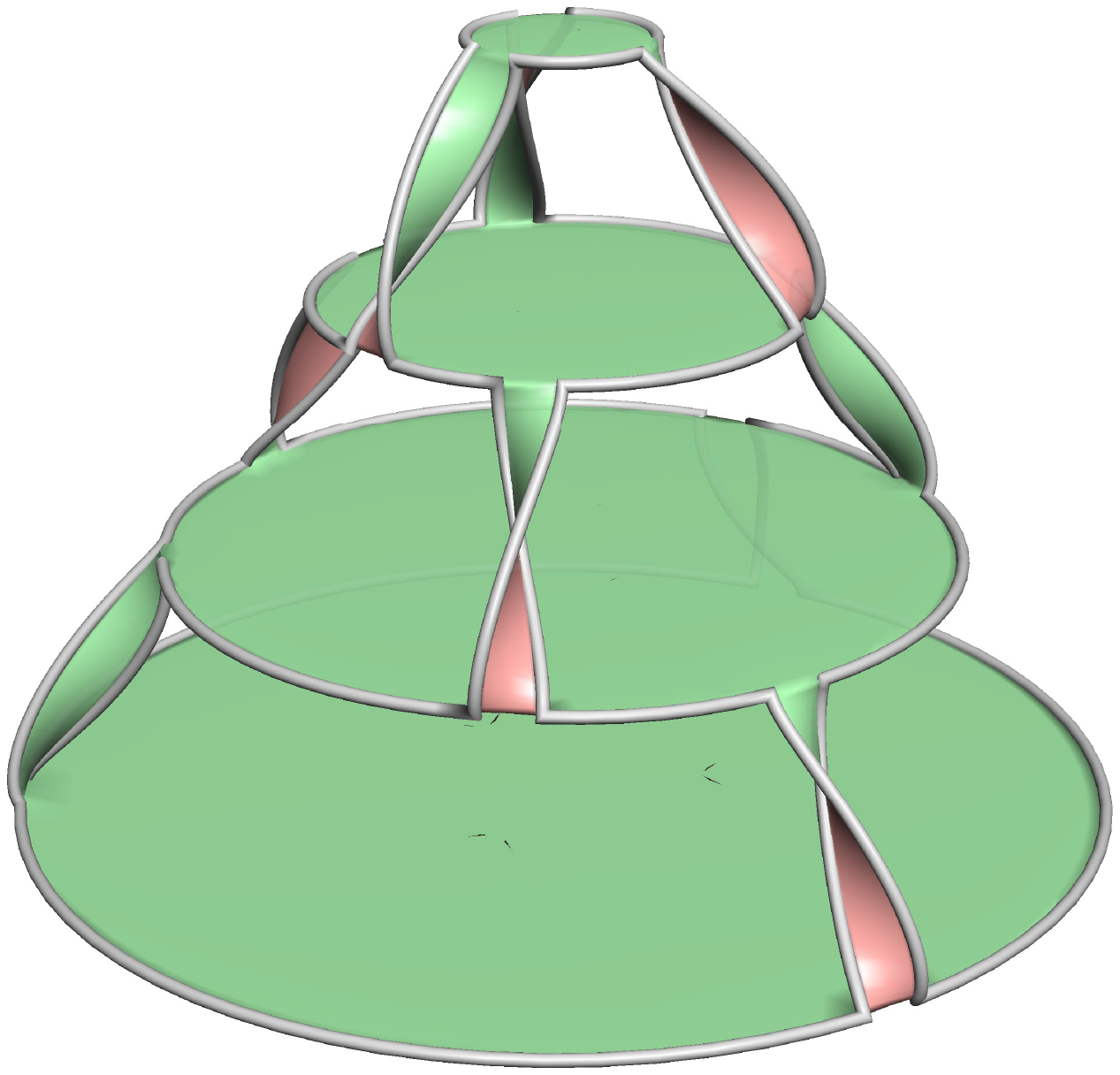}
    % Tex seems unhappy with the png version of the file, but likes the pdf for whatever reason
    \end{minipage}   
    \caption{An example of the ``cake surface'' (center and right) for the spiral knot $S(4,3,(1,-1,1))$ (left), which is a Seifert surface for the knot and is made up of disks and half-twist bands connecting each disk to an adjacent disk. The bands are labeled $h_{i,j}$ according to their vertical and horizontal position. The rightmost ``cake surface'' image was generated using SeifertView \cite{SeifertView}.  \label{fig:cake}}
\end{figure}

The cake surface for $S(p,q,\varepsilon)$ has $p$ disks corresponding to Seifert circles and $q$ half-twist bands connecting each disk to an adjacent disk for $q(p-1)$ total twist bands. The half-twist bands of the cake surface are labeled $h_{ij}$ according to their vertical position $1 \leq i \leq p-1$ and horizontal position $1 \leq j \leq q$.
Because the braid form is periodic, the horizontal position of the band does not affect the direction of the twist of $h_{i,j}$ which is thus dictated by $\varepsilon_i$. 

\begin{definition} \label{def:basis}
The \textit{cake homology basis} of $S(p,q,\varepsilon)$ is the set of (homology classes of) loops $\alpha_{ij}$ on the cake surface of $S(p,q,\varepsilon)$ where $1 \leq i \leq p-1$ and $1 \leq j \leq q-1$, and $\alpha_{ij}$ passes up through the half-twist band $h_{ij}$ once, then down through the half-twist band $h_{i,j+1}$ once, and not through any other bands.\footnote{This is the basis developed by Collins \cite{Collins2016} for use in her Seifert matrix computation algorithm \cite{SeifertMatrixCompuations}.}
\end{definition}

\begin{lemma}\label{lem:basis}
    The cake homology basis is a homology basis for the cake surface of $S(p,q,\varepsilon)$.
\end{lemma}

\begin{figure}[h]
    \begin{minipage}[c]{0.3\linewidth}
    \includegraphics[width=\linewidth,]{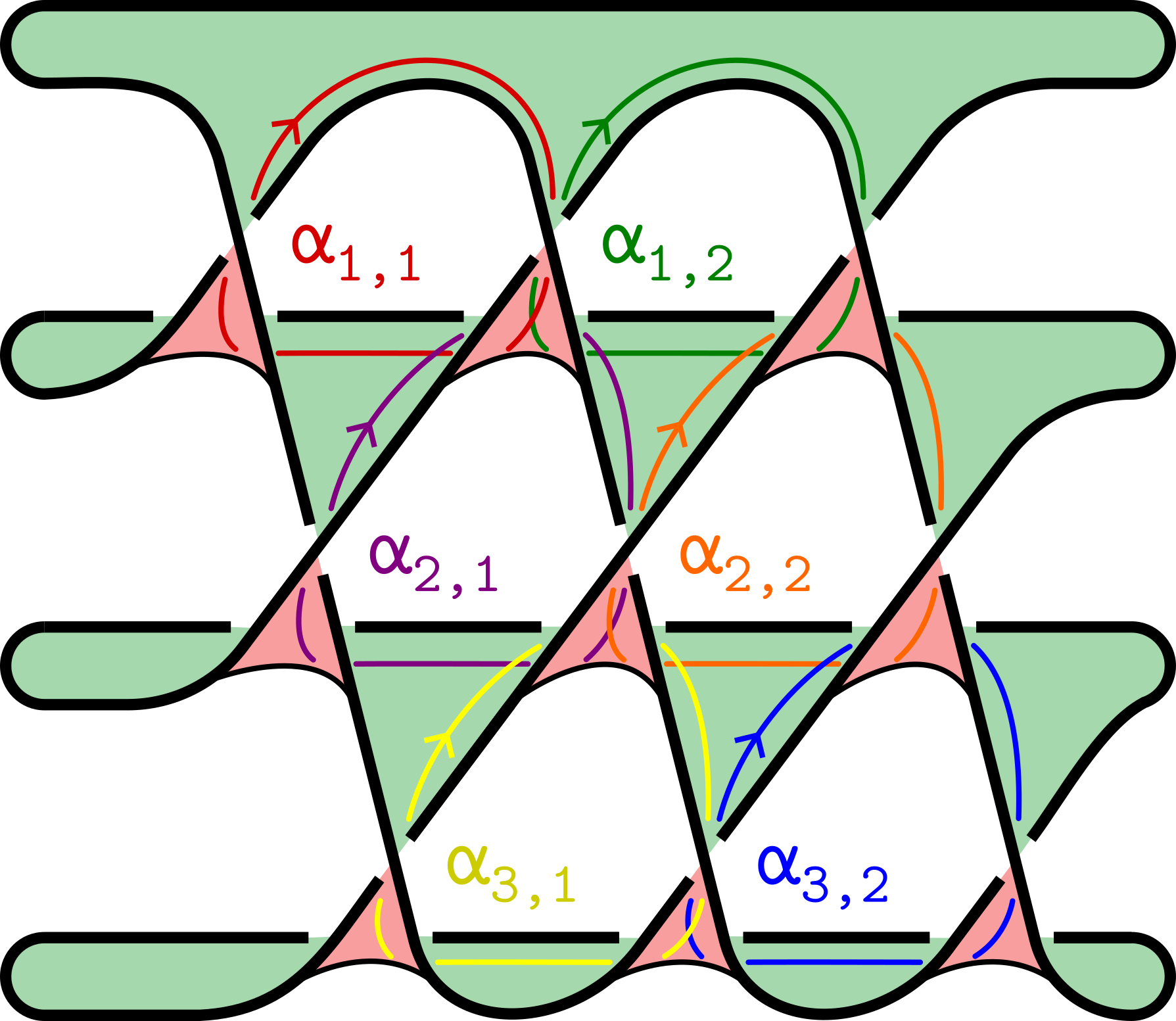}
    \end{minipage}
    \hfill
    \begin{minipage}[c]{0.3\linewidth}
    \includegraphics[width=\linewidth,]{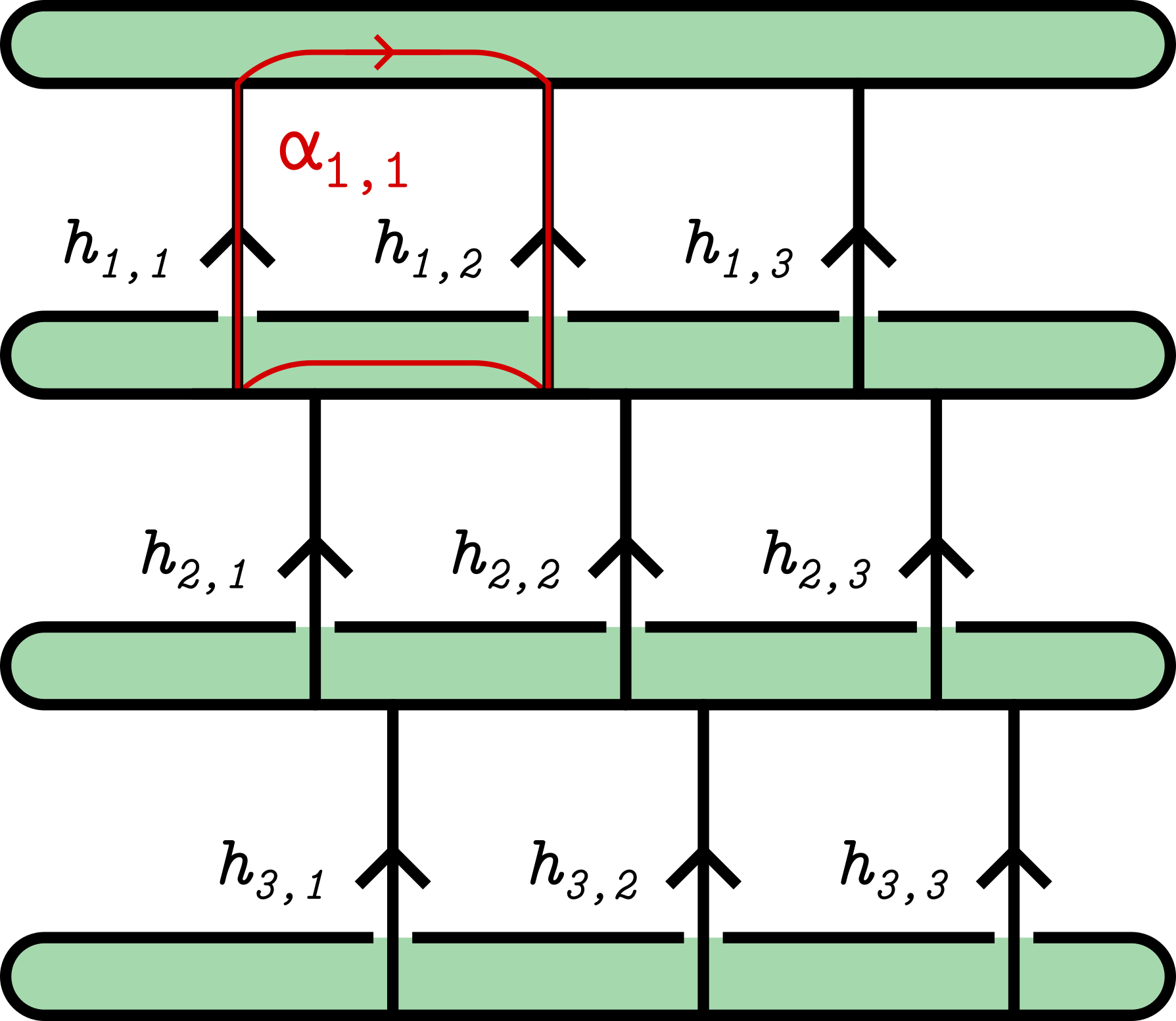}
    \end{minipage}
    \hfill
    \begin{minipage}[c]{0.3\linewidth}
    \includegraphics[width=\linewidth,]{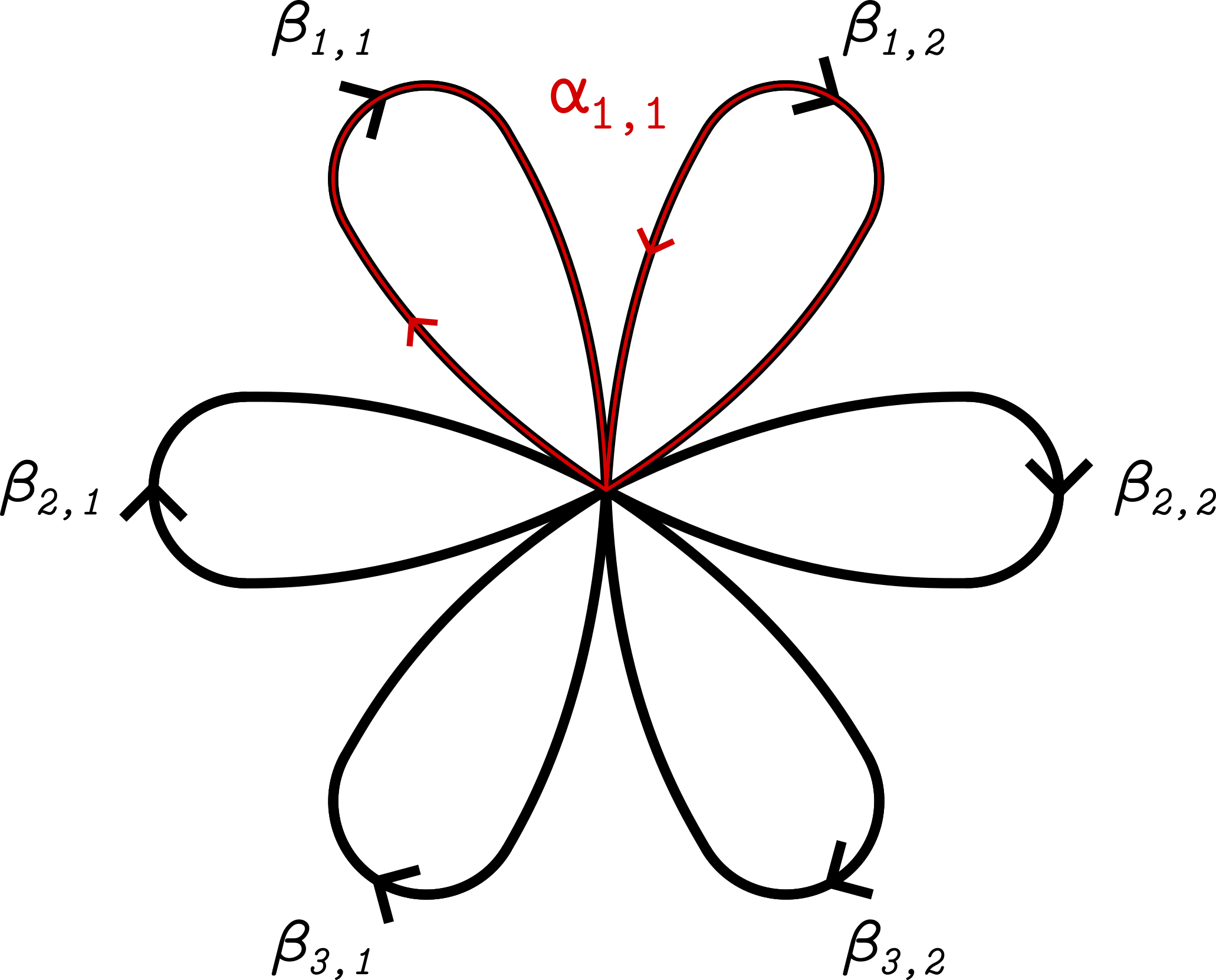}
    \end{minipage}
    \caption{Finding a homology basis for the cake surface. The left shows an example of the cake homology basis $\alpha_{ij}$ on a cake surface. The center and right show a deformation retraction from the cake surface to a wedge of circles, as described in the proof of \autoref{lem:basis}. \label{fig:retract}}
\end{figure}

\begin{proof}
Let $A = \{\alpha_{ij}\}$ be the cake homology basis of $S(p,q,\varepsilon)$, with the loops oriented clockwise as depicted in \autoref{fig:retract} (left), and let $X$ denote the cake surface. Let $Y$ be the wedge of $(p-1)(q-1)$ circles. To show that $A$ generates the first homology of $X$, we describe a deformation retraction of $X$ onto $Y$. First retract every band $h_{ij}$ to a line as shown in \autoref{fig:retract} (left)-(center). The set of $p$ disks along with the half-twist bands $h_{iq}$ form a contractible subspace of $X$. This subspace can be retracted to a point as in \autoref{fig:retract} (center)-(right). These retractions give us a deformation retraction $f \colon X \to Y$. Note that $H_1(Y)$ has rank $(p-1)(q-1)$. 

For each half-twist band $h_{ij}$, let $\beta_{ij}$ be (the homology class of) a loop in $Y$ corresponding to the image of $h_{ij}$ under $f$ with orientation as depicted in \autoref{fig:retract} (center). Let $B = \{\beta_{ij}\}$. Note that the $\beta_{iq}$ are trivial loops in $Y$ for every $i$ and that there is a bijective correspondence between non-trivial $\beta_{ij} \in B$ and factors in the wedge of circles $Y$. It follows that $B$ is a homology basis for $Y$.  

% Let $A_* = \{a_{ij}\} \subset H_1(X)$ and $B_* = \{b_{ij}\} \subset H_1(Y)$ be the sets of equivalence classes of loops in $A$ and $B$ respectively and
Let $f_* \colon H_1(X) \to H_1(Y)$ be the homomorphism induced by $f$. Note that $f_*(\beta_{iq})=0$ for every $i$. 
Since $\alpha_{ij}$ traverses up $h_{ij}$ then down $h_{i,j+1}$, we have $f_*(\alpha_{ij}) = \beta_{ij} - \beta_{i,j+1}$. Any $\beta_{ij}$ can therefore be written as 
\[\begin{alignedat}{1}\beta_{ij} &= (\beta_{ij} - \beta_{i,j+1}) + (\beta_{i,j+1}-\beta_{i,j+2}) + \cdots + (\beta_{i,q-2}-\beta_{i,q-1}) + (\beta_{i,q-1} + 0) \\
&= f_*(\alpha_{ij}) + f_*(\alpha_{i,j+1}) + \cdots + f_*(\alpha_{i,q-2}) + f_*(\alpha_{i,q-1} ).
\end{alignedat}\]

Thus the subgroup generated by $f_*(A)$ contains $B$, which is a homology basis for $Y$, so  $f_*(A)$ generates $H_1(Y)$. But $f$, being a deformation retraction, is a homotopy equivalence. So the induced map $f_* \colon H_1(X) \to H_1(Y)$ is an isomorphism. It follows that $H_1(X)$ has rank $(p-1)(q-1)$ and $(f_*)^{-1}(f_*(A)) = A$ generates $(f_*)^{-1}(H_1(Y)) = H_1(X)$.
% Now note that $A$ has $(p-1)(q-1)$ elements. Thus,
% \[\text{rk}(A_*) \leq (p-1)(q-1) = \text{rk}(H_1(Y)) = \text{rk}((f_*)^{-1}(H_1(Y)))= \text{rk}(H_1(X)).\] 
% Since $A_*$ generates $H_1(X)$ but does not have higher rank, it follows that $A_*$ is a homology basis for $X$. 
Since $|A|=(p-1)(q-1)=\text{rk}(H_1(X))$, it follows that $A$ is a homology basis for $X$.
\end{proof}

Because the cake homology basis has $(p-1)(q-1)$ elements, the cake surface has genus $(p-1)(q-1)/2$, giving us the following upper bound.

\begin{corollary}\label{lem:upgenus}
The genus $g$ of a spiral knot $S(p,q,\varepsilon)$ satisfies 
\[g(S(p,q,\varepsilon)) \leq \frac{(p-1)(q-1)}{2}.\]
\end{corollary} 

We can now use the cake homology basis of a spiral knot to find general forms for the Seifert matrix  and Alexander polynomial of a spiral knot. One consequence of this will be to achieve the corresponding lower bound for the genus (see \autoref{Ther:genus}). Throughout this paper, we will adopt the convention that blank entries of a matrix are 0. 

\begin{lemma}\label{lem:matrix}
The Seifert matrix $M \in M_{(p-1)(q-1)}(\mathbb{Z})$ for the cake surface of $S(p,q,\varepsilon)$ has the $(q-1) \times (q-1)$ block tri-diagonal form\footnote{This form for $M$ was discovered using \cite{SeifertMatrixCompuations}.}
\[M = \begin{bmatrix} 
-(K+L)& K & & \\
L & \ddots & \ddots & \\
& \ddots & \ddots & K \\
 & & L & -(K+L)
\end{bmatrix}
\]
where $K = [k_{ij}]$ and $L = [l_{ij}]$ are $(p-1)\times(p-1)$ submatrices defined by 
\[k_{ij} = \begin{cases} -1, & j = i,\ \varepsilon_i = -1 \\
0, & \text{otherwise} \end{cases} \text{ and } l_{ij} = \begin{cases} 1, & j = i,\ \varepsilon_i = 1 \\
-1, & j = i+1 \\
0, & \text{otherwise} \end{cases}.\]    
\end{lemma}

\begin{proof}
Consider $S(p,q,\varepsilon)$ and its cake homology basis given by \autoref{def:basis}. Then the Seifert matrix $M$ for the cake surface of $S(p,q,\varepsilon)$ can be expressed in the $(q-1) \times (q-1)$ block form
\[M = \begin{bmatrix} 
M_{11} & \dots & M_{1,q-1} \\
\vdots & \ddots & \vdots \\
M_{q-1,1} & \dots & M_{q-1,q-1} \\
\end{bmatrix}\]
where $M_{nm}$ is a $(p-1) \times (p-1)$ matrix defined by 
\[M_{nm} = [\lk(\alpha_{i,n}, \alpha_{j,m}^+)]  = \begin{bmatrix} 
\lk(\alpha_{1,n},\alpha_{1,m}^+) & \dots & \lk(\alpha_{1,n},\alpha_{p-1,m}^+)\\
\vdots & \ddots & \vdots \\
\lk(\alpha_{p-1,n},\alpha_{1,m}^+) & \dots & \lk(\alpha_{p-1,n},\alpha_{p-1,m}^+)\\
\end{bmatrix}\]
and $\lk(\alpha,\alpha')$ denotes the linking number of loops $\alpha$ and $\alpha'$. See \autoref{fig:loops}.

\begin{figure}[h]
\centering
    \includegraphics[width=.5\linewidth]{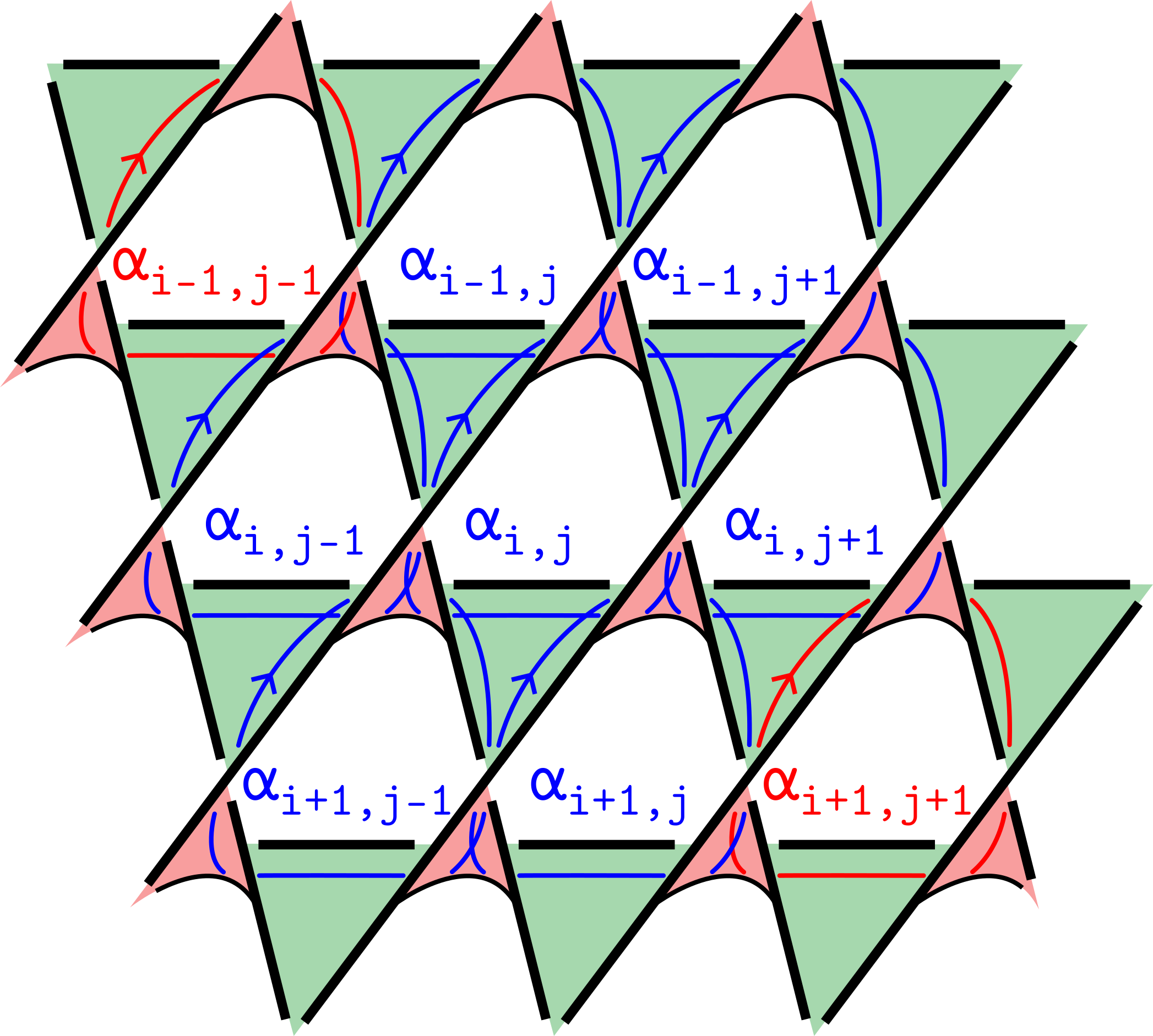}
    \caption{A portion of the cake homology basis on a cake surface, used for reference in the construction of the Seifert matrix in the proof of \autoref{lem:matrix}. \label{fig:loops}}
\end{figure}

Note that by construction, the loop $\alpha_{ij}$ will only intersect the loops $\alpha^+_{ij}$, $\alpha^+_{i,j\pm 1}$, $\alpha^+_{i\pm1,j}$, and $\alpha^+_{i \mp 1, j \pm 1}$. 
It follows that $\lk(a_{ij}, a^+_{i',j'})$ will be zero when $|j-j'| > 1$, and therefore $M_{nm} = 0$ if $|n-m| > 1$. Since half-twist bands with the same vertical position $i$ twist in the same direction, it follows that $\lk(\alpha_{i,j},\alpha_{i',j'}^+) =  \lk(\alpha_{i,j+1},\alpha_{i',j'+1}^+)$, which implies that $M_{nm} = M_{n+1,m+1}$. Thus $M$ has the form 
\[M = \begin{bmatrix} 
M_{11} & M_{12} & & \\
M_{21} & \ddots & \ddots & \\
& \ddots & \ddots & M_{12} \\
 & & M_{21} & M_{11}
\end{bmatrix}
.\]
We will now compute $M_{21}$, $M_{11}$, and $M_{12}$:
 \[(M_{21})_{ij} = \lk(\alpha_{i,2},\alpha_{j,1}^+) = \begin{cases}
\lk(\alpha_{i,2},\alpha_{i,1}^+),  & j = i, \varepsilon_i = 1 \\
\lk(\alpha_{i,2},\alpha_{i,1}^+),  & j = i, \varepsilon_i = -1 \\
\lk(\alpha_{i,2},\alpha_{i+1,1}^+),  & j = i+1 \\
\lk(\alpha_{i,2},\alpha_{j,1}^+),  & \text{otherwise}
\end{cases} 
= \begin{cases}
1, & j = i, \varepsilon_i = 1 \\
0, & j = i, \varepsilon_i = -1 \\
-1, & j = i+1 \\
0, & \text{otherwise},
\end{cases}\]
\[(M_{11})_{ij} = \lk(\alpha_{i,1},\alpha_{j,1}^+) = \begin{cases}
 \lk(\alpha_{i,1},\alpha_{i-1,1}^+),  & j = i-1 \\
 \lk(\alpha_{i,1},\alpha_{i,1}^+), & j = i, \varepsilon_i = 1 \\
 \lk(\alpha_{i,1},\alpha_{i,1}^+),  & j = i, \varepsilon_i = -1 \\
 \lk(\alpha_{i,1},\alpha_{i+1,1}^+),  & j = i+1 \\
 \lk(\alpha_{i,1},\alpha_{j,1}^+),  & \text{otherwise}
\end{cases}
= \begin{cases}
0, & j = i - 1 \\
-1, & j = i , \varepsilon_i = 1 \\
1, & j = i , \varepsilon_i = -1 \\
1, & j = i + 1 \\
0, & \text{otherwise},
\end{cases}\]
 \[(M_{12})_{ij} = \lk(\alpha_{i,1},\alpha_{j,2}^+) = \begin{cases}
 \lk(\alpha_{i,1},\alpha_{i-1,2}^+), & j = i - 1 \\
 \lk(\alpha_{i,1},\alpha_{i,2}^+), & j = i, \varepsilon_i = 1 \\
 \lk(\alpha_{i,1},\alpha_{i,2}^+), & j = i, \varepsilon_i = -1 \\
 \lk(\alpha_{i,1},\alpha_{j,2}^+), & \text{otherwise}
\end{cases}
= \begin{cases}
0, & j = i - 1 \\
0, & j = i, \varepsilon_i = 1 \\
-1, & j = i, \varepsilon_i = -1 \\
0, & \text{otherwise}.
\end{cases}\]
Setting $L = M_{21}$ and $K = M_{12}$ gives the result. 
\end{proof}

\subsection{Main theorem}
We now present and prove our main theorem. 

\Alex

The simple recursive structure of this formula allows us to more easily track the effects of changes in $q$ and $\varepsilon$ than directly computing the Alexander polynomial from the Seifert matrix, illuminating several properties about both the degree and coefficients of the polynomial, detailed in \autoref{subsec:properties}. Note that although it may not seem so at first glance, the $\mathcal{C}_k$ do depend heavily on $\varepsilon$, as their recursive definition depends on $\mu$, which in turn depends on $\varepsilon$. As a convention, we define $\mu(i)=0$ when $i\notin\{1,\dots,p-1\}$ in order to make certain arguments more concise.

Our results throughout this paper are phrased for the case of knots, but \autoref{Ther:Alex} also holds for the single-variable Alexander polynomials of spiral links. 
We discuss this further in \autoref{sec:future}.

\begin{proof} 
Consider a tri-diagonal block matrix with no corners, that is, a matrix of the form 
\[\begin{bmatrix} 
A_1 & B_1 & & \\
C_1 & \ddots & \ddots & \\
& \ddots & \ddots & B_{n-1} \\
 & & C_{n-1} & A_{n}
\end{bmatrix}
\]
where $A_i,B_i,C_i$ are $m\times m$ matrices. A result of Molinari \cite{Molinari_2008} states that the determinant of this matrix is 
\[(-1)^{mn} \det (R_n) \det (B_1\cdot \cdot \cdot B_{n-1})\]
where $R_n$\footnote{Molinari denotes $R_n$ as $T(n)_{11}$.} satisfies the recurrence relation
\[R_k = -B_k^{-1}A_kR_{k-1} - B_k^{-1}C_{k-1}R_{k-2}\]
with $R_0 = I_{m}$ and $R_1 = -B_1^{-1}A_1$. 

Let $M$ be the Seifert matrix for the cake homology basis of $S(p,q,\varepsilon)$. \autoref{lem:matrix} shows that $M-tM^T$ has the form
\[M-tM^T = \begin{bmatrix} 
-(K+L)+t(K+L)^T& K-tL^T & & \\
L-tK^T & \ddots & \ddots & \\
& \ddots & \ddots & K-tL^T \\
 & & L-tK^T & -(K+L)+t(K+L)^T
\end{bmatrix}.
\]
Note that $M-tM^T$ is a tri-diagonal block matrix with no corners, with  
\begin{align*}
    A_1 &= \cdots = A_{q-1} = -(K+L)+t(K+L)^T
    \\B_1 &= \cdots = B_{q-2} = K-tL^T
    \\C_1 &= \cdots = C_{q-2} = L-tK^T
\end{align*}
(where here $m=p-1$ and $n = q-1$). We can substitute these submatrices into the equation for $R_k$ to get 
\[\Delta(t) = \det(M-tM^T) = (-1)^{(p-1)(q-1)}\det(R_{q-1})\det((K-tL^T)^{q-2}).\]
Let $\mathbf{A} = B_k^{-1}C_k = (K-tL^T)^{-1}(L-tK^T)$.
It follows that for all $k=1,\ldots,q-1$,
\[\begin{alignedat}{1}
-B_k^{-1}A_k &= -(K-tL^T)^{-1}(-K-L+t(K+L)^T) \\
&= (K-tL^T)^{-1}(K+L-t(K+L)^T) \\
&= (K-tL^T)^{-1}(K-tL^T+L-tK^T) \\
&= (K-tL^T)^{-1}(K-tL^T)+(K-tL^T)^{-1}(L-tK^T) \\
&= I+\mathbf{A} .\\
\end{alignedat}\]

We induct on $k$ to show that $R_k=I+\mathbf{A}+ \cdots +\mathbf{A}^k$.
This holds for $k=0,1$ since $R_0 = I$ and $R_1 = -B_1^{-1}A_1 = I+\mathbf{A}$. 
Suppose that $R_i = I + \mathbf{A} + \cdots + \mathbf{A}^i$ for all $i < k$. Then 
\[\begin{alignedat}{1}
R_k &= -B_k^{-1}A_k R_{k-1} - B_k^{-1}C_{k-1}R_{k-2} \\
&= (I+\mathbf{A})R_{k-1}-\mathbf{A}R_{k-2} \\
&= R_{k-1} + \mathbf{A}(R_{k-1} - R_{k-2}) \\
&= I + \mathbf{A} + \cdots + \mathbf{A}^{k-1} + \mathbf{A}(\mathbf{A}^{k-1})\\
&= I + \mathbf{A} + \cdots + \mathbf{A}^k. \\
\end{alignedat}\]
Thus $R_{q-1} = I + \mathbf{A}+ \cdots + \mathbf{A}^{q-1}$. Note that $K-tL^T$ is a lower triangular matrix with $-t$'s and $-1$'s on the diagonal. So $\det((K-tL^T)^{q-2})= \pm t^n$ for some $n \in \mathbb{N}$. Since the Alexander polynomial is defined up to multiplication by $\pm t^\ell$ for $\ell \in \mathbb{Z}$, we have
\[\Delta(t) = \det(M-tM^T)=\det(R_{q-1}) = \det(I+ \mathbf{A} + \cdots + \mathbf{A}^{q-1}). \]

Let $\mathbf{A} = P J P^{-1}$, where $J$ is the Jordan normal form of $\mathbf{A}$. Then we have
\[\begin{alignedat}{1}
\Delta(t) &= \det(I + \mathbf{A}+ \cdots + \mathbf{A}^{q-1}) \\
&= \det(P J^0 P^{-1} + PJP^{-1}+ \cdots + P J^{q-1} P^{-1}) \\
&= \det(P(J^0 + \cdots +J^{q-1} )P^{-1}) \\
&= \det(J^0 + \cdots +J^{q-1} ). \\
\end{alignedat}\]
Let $\lambda_m$ for $1 \leq m \leq p-1$ be the eigenvalues of $\mathbf{A}$. Then the matrix $J^0 + \cdots +J^{q-1}$ is an upper triangular matrix with diagonal entries $[J^0+\cdots+J^{q-1}]_{ii}=\lambda_i^0 + \cdots + \lambda_i^{q-1}$. The determinant is the product of these entries:
\[
\Delta(t) = \det(J^0 + \cdots +J^{q-1} ) = \prod\limits_{m = 1}^{p-1} (\lambda_m^0 + \cdots + \lambda_m^{q-1}) = \prod\limits_{m = 1}^{p-1} \sum_{j=0}^{q-1}\lambda_m^j.
\]
Noting that the polynomial $f(y)=\sum_{j=0}^{q-1} y^j$ has $q-1$ roots given by $e^\frac{2 \pi \ell i}{q}$ where  $\ell \in \mathbb{Z}$ and $1 \le \ell \le q-1$, and the characteristic polynomial $\chi(x)$ of $\mathbf{A}$ has roots given by the eigenvalues $\lambda_m$ where $m \in \mathbb{Z}$ and $1 \le m \le p-1$, we can factor both of these polynomials.
\begin{align*}
    f(y) &= (y-e^\frac{2\pi i}{q})(y-e^\frac{2\pi 2 i}{q})\cdots(y-e^\frac{2\pi (q-1) i}{q}) = \prod_{\ell=1}^{q-1}(y-e^\frac{2\pi \ell i}{q})\\
    \chi(x) &= (x-\lambda_1)(x-\lambda_2)\cdots(x-\lambda_{p-1}) = \prod_{m=1}^{p-1}(x-\lambda_m)
\end{align*}

Thus, we can simplify the determinant of $J^0+\cdots+J^{q-1}$ to
\[
\prod\limits_{m = 1}^{p-1} \sum_{j=0}^{q-1}\lambda_m^j
=\prod\limits_{m = 1}^{p-1} f(\lambda_m)
= \prod\limits_{m = 1}^{p-1} \prod_{\ell = 1}^{q-1} (\lambda_m - e^\frac{2\pi \ell i}{q})
=(-1)^{(p-1)(q-1)}\prod\limits_{m = 1}^{p-1} \prod_{\ell = 1}^{q-1} (e^\frac{2\pi \ell i}{q} - \lambda_m).
\]
Because $\gcd(p,q)=1$, at least one of $p$ or $q$ must be odd, so $(p-1)(q-1)$ is even.\footnote{Note that even if we did not have $\gcd(p,q)=1$, that is, if we had a spiral \textit{link}, this would not cause a problem here as the Alexander polynomial is defined up to multiplication by $\pm 1$.} Furthermore, we can switch the order of the products to get
\[
(-1)^{(p-1)(q-1)}\prod\limits_{m = 1}^{p-1} \prod_{\ell = 1}^{q-1} (e^\frac{2\pi \ell i}{q} - \lambda_m)
= \prod_{\ell = 1}^{q-1} \prod\limits_{m = 1}^{p-1} (e^\frac{2\pi \ell i}{q} - \lambda_m)
= \prod_{\ell = 1}^{q-1} \chi(e^\frac{2\pi \ell i}{q}).
\]
% This is the point at which the footnote breaks
In \autoref{sec:lemmaproofs} we will prove the following lemma. 
\begin{restatable}{lemma}{char}
\label{lem:char}%
The characteristic polynomial $\chi$ of $\mathbf{A}$ is given by $\chi(x) = \mathcal{C}_{p-1}(x,t)$.
\end{restatable}

Therefore by \autoref{lem:char} we have
\[
\Delta(t) = \prod_{\ell = 1}^{q-1} \chi(e^\frac{2\pi \ell i}{q})
= \prod_{\ell = 1}^{q-1} \mathcal{C}_{p-1}(e^\frac{2\pi\ell i}{q},t),
\] 
and we are done.
\end{proof}

We conclude this subsection with an example of how the recursive formula from \autoref{Ther:Alex} may be applied in practice to compute the Alexander polynomial of a spiral knot.

\begin{example}
We compute the Alexander polynomial of $6_2=S(5,2,(1,1,1,-1))$. First, for each $\varepsilon_k$, we note the value of $\mu(k)$.
\begin{center}
    \begin{tabular}{c|c|c|c|c}
    $k$ & $1$ & $2$ & $3$ & $4$ \\
    \hline
    $\mu(k)$ & $1$ & $1$ & $1$ & $t$ \\
    \end{tabular}
\end{center}

Now we find $\mathcal{C}_k(x,t)$ for $k=0, \dots, 4$.
 \[\begin{alignedat}{1}
        \mathcal{C}_0(x,t) &= 1  \\
        \mathcal{C}_1(x,t) &= \frac{\mu(1)^2}{t}+x = \frac{1}{t}+x \\
        \mathcal{C}_2(x,t) &=  \left(\frac{\mu(2)^2}{t}+x\right)\mathcal{C}_{1} - \left(\frac{\mu(1)\mu(2)x}{t}\right)\mathcal{C}_{0} = \left(\frac{1}{t}+x\right)^2 - \frac{x}{t} % = \frac{1}{t^2} + \frac{x}{t} + x^2 
        \\
        \mathcal{C}_3(x,t) &=  \left(\frac{\mu(3)^2}{t}+x\right)\mathcal{C}_{2} - \left(\frac{\mu(2)\mu(3)x}{t}\right)\mathcal{C}_{1} =  \left(\frac{1}{t}+x\right) \left( \left(\frac{1}{t}+x\right)^2 - \frac{x}{t} \right) - \frac{x}{t} \left(\frac{1}{t}+x\right) \\
        &=  \left(\frac{1}{t}+x\right)^3 - \frac{2x}{t} \left(\frac{1}{t}+x\right) \\
        \mathcal{C}_4(x,t) &=  \left(\frac{\mu(4)^2}{t}+x\right)\mathcal{C}_{3} - \left(\frac{\mu(3)\mu(4)x}{t}\right)\mathcal{C}_{2} \\
        &= \left(t+x\right) \left( \left(\frac{1}{t}+x\right)^3 - \frac{2x}{t} \left(\frac{1}{t}+x\right) \right) - x \left(  \left(\frac{1}{t}+x\right)^2 - \frac{x}{t} \right) 
    \end{alignedat}\]

Then $\Delta_{6_2}(t) =  
\mathcal{C}_{4}\left(e^{\frac{2\pi}{2}i},t\right) = \mathcal{C}_{4}\left(-1,t\right)$ so we compute:
\begin{align*}
\mathcal{C}_{4}\left(-1,t\right) &= \left(t-1\right) \left( \left(t^{-1}-1\right)^3 + 2t^{-1} \left(t^{-1}-1\right) \right) + \left(  \left(t^{-1}-1\right)^2 + t^{-1} \right) \\
&= \left(t-1\right) \left( t^{-3} -3t^{-2} +3t^{-1} -1 + 2t^{-2} -2t^{-1} \right) + \left(  t^{-2}-2t^{-1}+1 + t^{-1} \right) \\
&= \left(t-1\right) \left( t^{-3} -t^{-2} + t^{-1} -1 \right) + \left(  t^{-2}-t^{-1}+1 \right) \\
&= -t + 3 -3t^{-1} + 3t^{-2} - t^{-3} \\ &\equiv t^4 -3t^3 +3t^2 -3t+1.
\end{align*}

\end{example}

\subsection{Properties of the Alexander polynomial} \label{subsec:properties}
To find general properties of $\Delta (t)$, we first look into the general form of the $\mathcal{C}_k(x,t)$ defined recursively in \autoref{Ther:Alex}. We define the following notation to aid in this endeavor.

\begin{itemize}
    \item  Let $\alpha_k$ be the number of $1$'s and $\beta_k$ the number of $-1$'s in the first $k$ terms of $\varepsilon$. Then $\alpha_{p-1}$ and $\beta_{p-1}$ are the numbers of $1$'s and $-1$'s, respectively, in $\varepsilon$. We define $\alpha_0 = \beta_0 = 0$. Note that $\alpha_k + \beta_k = k$.
    \item Call a string of consecutive positive $1$'s in $\varepsilon$ a \textit{positive block} and a row of consecutive $-1$'s a \textit{negative block}. Let $A_k$ and $B_k$ be the number of positive and negative blocks in the first $k$ terms of $\varepsilon$, respectively.
    \item Let $\gamma_k$ be the number of times that $\varepsilon$ changes sign in its first $k$ entries. Note that $\gamma_k = A_k+B_k-1$.
\end{itemize}

For example, if $\varepsilon = (1,1,-1,-1,-1,1,1,\ldots)$ then $\alpha_7=4$, $\beta_7=3$, $A_7=2$, $B_7=1$, and $\gamma_7=2$. We can then describe the general form of $\mathcal{C}_k(x,t)$ as follows.

\begin{restatable}{lemma}{cform}
\label{lem:cform}%
$\mathcal{C}_k(x,t)$ has the form
    \[\mathcal{C}_k(x,t) = \frac{1}{t^{\alpha_k}}\left(\mathcal{C}_k^0(x) + \mathcal{C}_k^1(x)t + \dots + \mathcal{C}_k^k(x)t^k\right) = \sum_{i = 0}^k\mathcal{C}_k^i(x)t^{i-\alpha_k},\]
    where $\mathcal{C}_k^0(x) = x^{\beta_k}$, $\mathcal{C}_k^1(x) = (A_kx-\gamma_k+\frac{B_k}{x})x^{\beta_k}$, and $\mathcal{C}_k^k(x) = x^{\alpha_k}$. 
\end{restatable}

Together, \autoref{Ther:Alex} and \autoref{lem:cform} give the following form for $\Delta(t)$:
\[
    \Delta(t) = \prod\limits_{\ell = 1}^{q-1} 
\mathcal{C}_{p-1}\left(e^\frac{2\pi \ell i}{q},t\right) = \prod\limits_{\ell = 1}^{q-1}\sum_{j = 0}^{p-1}\mathcal{C}_{p-1}^j(e^\frac{2\pi \ell i}{q})t^{j-\alpha_{p-1}}.
\]

We defer the proof of \autoref{lem:cform} to \autoref{sec:lemmaproofs}, but first we point out several properties that follow from this expression.

\begin{corollary} \label{cor:deg}
    The Alexander polynomial of the spiral knot $S(p,q,\varepsilon)$ has degree $(p-1)(q-1)$.
\end{corollary}
\begin{proof}
We know from \autoref{lem:cform} that $\mathcal{C}_{p-1}^{0}(e^\frac{2\pi \ell i}{q}) = (e^\frac{2\pi \ell i}{q})^{\beta_{p-1}} \neq 0$ and $\mathcal{C}_{p-1}^{p-1}(e^\frac{2\pi \ell i}{q}) = (e^\frac{2\pi \ell i}{q})^{\alpha_{p-1}} \neq 0$ for all $\ell$. Therefore, the lowest power of $t$ in $\Delta(t)$ comes from multiplying the lowest power of $t$ in each summand within the product to get $t^{(q-1)(-\alpha_{p-1})}$. Likewise, the highest power of $t$ in $\Delta(t)$ is $t^{(q-1)(p-1-\alpha_{p-1})}$.
The degree of the Alexander polynomial of $S(p,q,\varepsilon)$ is the degree of the highest power of $t$ minus the degree of the lowest power of $t$:
\[
    (q-1)(p-1-\alpha_{p-1}) - (q-1)(-\alpha_{p-1}) = (q-1)(p-1-\alpha_{p-1}+\alpha_{p-1}) = (q-1)(p-1).
\]
\end{proof}

In \autoref{subsec:invariants} we will use \autoref{cor:deg} to determine the genus of any spiral knot (see \autoref{Ther:genus}). The following two corollaries, which concern the first and second coefficients of the Alexander polynomial, will be useful for classifying low-complexity spiral knots in \autoref{subsec:examples}. We will additionally use \autoref{cor:coeff} to conclude that $S(p,q,\varepsilon) = S(q,p,\varepsilon')$ for some $\varepsilon'$ if and only if $S(p,q,\varepsilon)$ is a torus knot (see \autoref{Ther:swappq}). 

As the Alexander polynomial is symmetric, the coefficients of the highest and second highest powers of $t$ equal the coefficients of the lowest and second lowest powers, respectively, and in the proofs that follow we calculate the coefficients of the lower powers for the sake of ease of notation.
We denote the coefficients of the lowest and second lowest powers of $t$ by $\Delta^0$ and $\Delta^1$, respectively.

\begin{corollary} \label{cor:monic}
    All spiral knots have monic Alexander polynomial.
\end{corollary}

\begin{proof}
Recall that the Alexander polynomial is symmetric.
Collecting the coefficients of $t^{(q-1)(-\alpha_{p-1})}$ and then applying \autoref{lem:cform}, we get:
\[\Delta^0 = \prod_{\ell =1}^{q-1}\mathcal{C}_{p-1}^0(e^\frac{2 \pi \ell  i}{q}) = \prod_{\ell =1}^{q-1}e^\frac{2 \pi \ell  \beta_{p-1} i}{q} = e^\frac{2 \pi \left(\sum_{\ell =1}^{q-1}\ell \right) \beta_{p-1} i}{q } = e^\frac{2 \pi \frac{(q-1)(q)}{2} \beta_{p-1} i}{q} = e^{ \pi (q-1) \beta_{p-1} i} = \pm 1. \]
\end{proof}

\begin{corollary} \label{cor:coeff}
    The second coefficient of the Alexander polynomial of the spiral knot $S(p,q,\varepsilon)$ is $\pm (q\gamma_{p-1} + 1)$.
\end{corollary}

\begin{proof}
Recall that the Alexander polynomial is symmetric. Collecting the coefficients of $t^{-(q-1)\alpha_{p-1}+1}$ then applying \autoref{lem:cform}, we get the following string of equalities. Note that $\mathcal{C}_{p-1}^{0}(e^\frac{2\pi \ell i}{q}) = (e^\frac{2\pi \ell i}{q})^{\beta_{p-1}} \neq 0$. 
\begin{align*}
 \Delta^1 &= \sum_{j=1}^{q-1} \left( 
\mathcal{C}_{p-1}^1 (e^{\frac{2\pi j i}{q}})\prod_{\substack{\ell=1 \\ \ell \neq j}}^{q-1} \mathcal{C}_{p-1}^0 (e^{\frac{2\pi \ell i}{q}}) \right)
\\&=\sum_{j = 1}^{q-1} \left( \prod_{\ell =1}^{q-1}C_{p-1}^0(e^\frac{2 \pi \ell  i}{q}) \right) \frac{C_{p-1}^1(e^\frac{2 \pi j i}{q})}{C_{p-1}^0(e^\frac{2 \pi j i}{q})}
\\&=\sum_{j = 1}^{q-1} \Delta^0 \frac{C_{p-1}^1(e^\frac{2 \pi j i}{q})}{C_{p-1}^0(e^\frac{2 \pi j i}{q})}
\\&= \pm \sum_{j = 1}^{q-1} \frac{C_{p-1}^1(e^\frac{2 \pi j i}{q})}{e^\frac{2 \pi j \beta_{p-1} i}{q}}
\\& = \pm \sum_{j = 1}^{q-1} \frac{(A_{p-1}e^{\frac{2\pi j i}{q}} - \gamma_{p-1} + B_{p-1}e^{-\frac{2\pi j i}{q}} )(e^{\frac{2\pi j \beta_{p-1} i}{q}})}{e^\frac{2 \pi j \beta_{p-1} i}{q}}
\\&= \pm \sum_{j = 1}^{q-1}A_{p-1}e^{\frac{2\pi j i}{q}} - \gamma_{p-1} + B_{p-1}e^{-\frac{2\pi j i}{q}}  \\
&= \pm \left(A_{p-1}\sum_{j = 1}^{q-1}e^{\frac{2\pi j i}{q}} - \sum_{j = 1}^{q-1}\gamma_{p-1} + B_{p-1}\sum_{j = 1}^{q-1}e^{-\frac{2\pi j i}{q}} \right) \\
&= \pm\left(-A_{p-1} - (q-1)\gamma_{p-1} - B_{p-1} \right) \\
&= \mp\left( (q-1)\gamma_{p-1} +A_{p-1}+ B_{p-1} \right) \\
&= \mp(q\gamma_{p-1}+1).
\end{align*}
\end{proof}

As spiral knots are closures of homogeneous braids, a result of Stallings \cite{Stallings} implies spiral knots are fibered, hence have monic Alexander polynomials.
This reproves \autoref{cor:monic}.
To our knowledge, \autoref{cor:deg} and \autoref{cor:coeff} (and their consequences) are the first known obstructions to a knot being spiral which are not obstructions to a knot being periodic and fibered. 
Known periodicity obstructions include the Murasugi conditions \cites{Murasugi,DavisLivingston}, discussed in \autoref{sec:future}, and the Edmonds conditions \cite{Edmonds}, which we note in \autoref{subsec:invariants} can be used to give an alternative proof of \autoref{Ther:genus}.

Observe that, for example, the knot $8_{21}$ is a hyperbolic, periodic, fibered knot with all nontrivial periods of order 2 (see KnotInfo \cite{KnotInfo}). 
If $8_{21}$ was a spiral knot, then we would have $8_{21}=S(p,2,\varepsilon)$.
Its Alexander polynomial is $\Delta_{8_{21}}(t)=t^4-4t^3+5t^2-4t+1$.
We can employ each of \autoref{cor:deg} and \autoref{cor:coeff} to show $8_{21}$ is not spiral.

Applying \autoref{cor:deg}, we see that if $8_{21}$ was spiral, we would have $p=5$. Performing a census of all spiral knots $S(5,2,\varepsilon)$, as we do in \autoref{subsec:examples}, shows that $8_{21}$ is not spiral.
Instead of identifying spiral knots $S(5,2,\varepsilon)$, we could further apply \autoref{cor:coeff} to see that if $8_{21}$ was spiral, then $\varepsilon$ would satisfy $\gamma_{p-1}=
\tfrac{5}{2}$, which is impossible since $\gamma_{p-1}$ is an integer. 

\subsection{Proofs of technical lemmas}
\label{sec:lemmaproofs}
We will now prove \autoref{lem:char} and \autoref{lem:cform}.

\char*

\begin{proof} We first find a general form for the matrix $\mathbf{A}$. Let $B \vcentcolon = K-tL^T$ and $C \vcentcolon = L-tK^T$. We have $\mathbf{A} = B^{-1}C = (K-tL^T)^{-1}(L-tK^T)$, where $K = [k_{ij}]$ and $L = [l_{ij}]$ are as previously defined: 
\[k_{ij} = \begin{cases} -1, & j=i,\ \varepsilon_i = -1 \\
0, & \text{otherwise} \end{cases} \text{ and } l_{ij} = \begin{cases} 1, & j=i,\ \varepsilon_i = 1 \\
-1, & j = i + 1 \\
0, & \text{otherwise} \end{cases}.\] 
Recall that $\mu$ is defined as $\mu(i) = \begin{cases}
1, & \varepsilon_i = 1 \\
t, & \varepsilon_i = -1 \\
0, & \text{otherwise}
\end{cases}$. Then $B = [b_{ij}]$ and $C = [c_{ij}]$ are given by: 
\begin{align*}
b_{ij} &= k_{ij} - tl_{ji} = \begin{cases}
-\frac{t}{\mu(i)}, & j=i \\
t, & j = i-1 \\
0, & \text{otherwise}
\end{cases},
\\
c_{ij} &= l_{ij} - tk_{ji} = \begin{cases}
\mu(i), & j=i \\
-1, & j = i+1 \\
0, & \text{otherwise}
\end{cases}.
\end{align*}
Let $D = [d_{ij}]$ be defined by $d_{ij} = \begin{cases}
-\frac{1}{t}\prod\limits_{j \leq l \leq i} \mu(l), & j \leq i \\
0, & \text{otherwise}
\end{cases}.$

We show that $D = B^{-1}$ by calculating the $(i,j)$ entry of the product $BD$. 
\[\begin{alignedat}{1}
[BD]_{ij} = \sum\limits_{k=1}^{p-1} b_{ik}d_{kj} &= b_{i,i-1}d_{i-1,j} + b_{ii}d_{ij} \\
&= td_{i-1,j} - \frac{t}{\mu(i)}d_{ij} \\
&= \begin{cases}
\frac{1}{\mu(i)} \prod\limits_{i \leq l \leq i} \mu(l), & j=i \\
-\prod\limits_{j \leq l \leq i-1} \mu(l) + \frac{1}{\mu(i)} \prod\limits_{j \leq l \leq i} \mu(l), & j < i \\
0, & \text{otherwise} 
\end{cases} \\
&= \begin{cases}
1, & j=i \\
0, & \text{otherwise}
\end{cases} \\
\end{alignedat}\]

So we have $BD = I$. Since $B$ is invertible, we have $\mathbf{A} = B^{-1}C = DC$. The general form for $\mathbf{A} = [a_{ij}]$ is therefore given by the following.
\begin{align*}
a_{ij} = \sum\limits_{k=1}^{p-1} d_{ik}c_{kj} &= d_{ij}c_{jj} + d_{i,j-1}c_{j-1,j} \\
&= \mu(j)d_{ij} - d_{i,j-1} \\
&= \begin{cases}
\frac{1}{t}\prod\limits_{i \leq l \leq i} \mu(l), & j = i+1\\
-\frac{1}{t}\left(\mu(j)\prod\limits_{j \leq l \leq i} \mu(l) - \prod\limits_{j-1 \leq l \leq i} \mu(l)\right), & 1 \leq j \leq i \\
0, & \text{otherwise} \\
\end{cases} \\
&= \begin{cases}\frac{1}{t}\mu(i), & j = i+1\\
\frac{1}{t}\Big(\mu(j-1) - \mu(j) \Big)\prod\limits_{j \leq l \leq i} \mu(l),  & 1 \leq j \leq i \\
0, & \text{otherwise} \\
\end{cases} \\
\end{align*}
When $j < i$, note that 
\[a_{ij} = \frac{\mu(j-1)-\mu(j)}{t}\prod\limits_{j \leq l \leq i} \mu(l) 
= \mu(i)\left(\frac{\mu(j-1)-\mu(j)}{t}\prod\limits_{j \leq l \leq i-1} \mu(l)\right) = \mu(i)a_{i-1,j}.\] Additionally, when $j=i$, we have \[a_{jj}=\frac{\mu(j-1)-\mu(j)}{t}\mu(j).\]

We now want to find the characteristic polynomial of $\mathbf{A}$. Define $P_k$ to be the submatrix given by the first $k$ rows and columns of $xI-\mathbf{A}$. We take $P_0$ to be the empty matrix.
Given a square matrix $A$, define $\widehat{A}$ to be the submatrix obtained by deleting the second-to-last row and last column from $A$. If $A$ is a $1\times 1$ matrix, $\widehat{A}$ is the empty matrix. For the sake of convenience, we take the determinant of the empty matrix to be 1. The matrices $\widehat{P}_k$ will appear in our computation of the determinant of $xI-\mathbf{A}$ via cofactor expansion.

We can utilize the general form for $\mathbf{A}$ to write down the following matrices. We have 

\[
\renewcommand{\arraystretch}{1.3}
P_k = \left[\begin{array}{c|c}
  \begin{matrix}
  & & & & \\
  & &P_{k-1}& & \\
  & & & &\\
  \end{matrix}
  & 
  \begin{matrix}
  0 \\
  \vdots \\
  0 \\
  -a_{k-1,k}
  \end{matrix}
  \\
\hline
  \begin{matrix}
-\mu(k)a_{k-1,1} & \dots & -\mu(k)a_{k-1,k-2} &-\mu(k)a_{k-1,k-1}
  \end{matrix} &
   x-a_{kk}
\end{array}\right]
\]
so that
\[\renewcommand{\arraystretch}{1.3}
\widehat{P}_k = \left[\begin{array}{c|c}
  \begin{matrix}
  & & & & \\
  & &P_{k-2}& & \\
  & & & &\\
  \end{matrix}
  & 
  \begin{matrix}
  0 \\
  \vdots \\
  0 \\
  -a_{k-2,k-1}
  \end{matrix}
  \\
\hline
  \begin{matrix}
-\mu(k)a_{k-1,1} & \dots & -\mu(k)a_{k-1,k-2}
  \end{matrix} &
  -\mu(k)a_{k-1,k-1}
\end{array}\right] 
\]
and
\[\renewcommand{\arraystretch}{1.3}
P_{k-1} = \left[\begin{array}{c|c}
  \begin{matrix}
  & & & & \\
  & &P_{k-2}& & \\
  & & & &\\
  \end{matrix}
  & 
  \begin{matrix}
  0 \\
  \vdots \\
  0 \\
  -a_{k-2,k-1}
  \end{matrix}
  \\
\hline
  \begin{matrix}
-a_{k-1,1} & \dots & -a_{k-1,k-2}
  \end{matrix} &
  x-a_{k-1,k-1}
\end{array}\right].
\]
We would like to compare the matrices $\widehat{\widehat{P\, }}_k$ and $\widehat{P}_{k-1}$. These are given by
\[\renewcommand{\arraystretch}{1.3}
\widehat{\widehat{P\, }}_k = \left[\begin{array}{c|c}
  \begin{matrix}
  & & & & \\
  & &P_{k-3}& & \\
  & & & &\\
  \end{matrix}
  & 
  \begin{matrix}
  0 \\
  \vdots \\
  0 \\
  -a_{k-3,k-2}
  \end{matrix}
  \\
\hline
  \begin{matrix}
-\mu(k)a_{k-1,1} & \dots & -\mu(k)a_{k-1,k-3}
  \end{matrix} &
  -\mu(k)a_{k-1,k-2}
\end{array}\right] 
\]
and
\[\renewcommand{\arraystretch}{1.3}
\widehat{P}_{k-1} = \left[\begin{array}{c|c}
  \begin{matrix}
  & & & & \\
  & &P_{k-3}& & \\
  & & & &\\
  \end{matrix}
  & 
  \begin{matrix}
  0 \\
  \vdots \\
  0 \\
  -a_{k-3,k-2}
  \end{matrix}
  \\
\hline
  \begin{matrix}
-a_{k-1,1} & \dots & -a_{k-1,k-3}
  \end{matrix} &
  -a_{k-1,k-2}
\end{array}\right]. 
\]

We can see that $\widehat{\widehat{P\, }}_k$ is given by multiplying the bottom row of $\widehat{P}_{k-1}$ by $\mu(k)$. It follows that, for $k\geq 3$, $\det(\widehat{\widehat{P\, }}_k) = \mu(k)\det(\widehat{P}_{k-1})$. We cofactor expand along the last column of $\widehat{P}_k$ to get
\[\begin{alignedat}{1}
\det(\widehat{P}_k) &= -\mu(k)a_{k-1,k-1}\det(P_{k-2}) + a_{k-2,k-1} \det(\widehat{\widehat{P\, }}_k) \\
&= -\mu(k)a_{k-1,k-1}\det(P_{k-2}) + a_{k-2,k-1} \mu(k) \det(\widehat{P}_{k-1}).
\end{alignedat}\]
Then, we cofactor expand along the last column of $P_{k-1}$ and rearrange the equation to get
\[\begin{alignedat}{2}
&& \quad \det(P_{k-1}) &= (x-a_{k-1,k-1})\det(P_{k-2}) + a_{k-2,k-1} \det(\widehat{P}_{k-1}) \\
\Leftrightarrow && \quad a_{k-2,k-1} \det(\widehat{P}_{k-1}) &= \det(P_{k-1}) -  (x-a_{k-1,k-1})\det(P_{k-2}). & 
\end{alignedat}\]
Plugging $a_{k-2,k-1} \det(\widehat{P}_{k-1})$ into $\det(\widehat{P}_k)$ gives
\[\begin{alignedat}{1}
\det(\widehat{P}_k) &= -\mu(k)a_{k-1,k-1}\det(P_{k-2}) + \mu(k)\big(\det(P_{k-1}) - (x-a_{k-1,k-1})\det(P_{k-2})\big) \\
&= \mu(k)\big(\det(P_{k-1}) - (x-a_{k-1,k-1})\det(P_{k-2}) - a_{k-1,k-1}\det(P_{k-2})\big) \\
&= \mu(k)\big(\det(P_{k-1}) - x\det(P_{k-2}) \big).
\end{alignedat}\]
By cofactor expanding along the last column of $P_k$, we can see that
\[\begin{alignedat}{1}
\det(P_k)={}& (x-a_{kk})\det(P_{k-1}) + a_{k-1,k}\det(\widehat{P}_k) \\
={}& (x-a_{kk})\det(P_{k-1}) + a_{k-1,k}\mu(k)\big(\det(P_{k-1})-x\det(P_{k-2})\big) \\
={}& (x-a_{kk}+ a_{k-1,k}\mu(k))\det(P_{k-1}) - a_{k-1,k}\mu(k)x\det(P_{k-2}) \\
={}& \left(x-\frac{\mu(k-1)-\mu(k)}{t}\mu(k)+ \frac{\mu(k-1)}{t}\mu(k)\right)\det(P_{k-1}) - \frac{\mu(k-1)}{t}\mu(k)x\det(P_{k-2}) \\
={}& \left(\frac{\mu(k)^2}{t}+x\right)\det(P_{k-1}) - \frac{\mu(k-1)\mu(k)x}{t}\det(P_{k-2}). \\
\end{alignedat}\]

Thus, $\det(P_k)$ satisfies the same recurrence relation as $\mathcal{C}_k(x,t)$ for $k\geq 3$. We check that $\det(P_k)=\mathcal{C}_k(x,t)$ for $0\leq k\leq 2$. By convention, $\det(P_{0}) = 1 = \mathcal{C}_0(x,t)$. When $k=1$, \[\det(P_1)=x-a_{11}=x+\frac{\mu(1)^2}{t}=\mathcal{C}_1(x,t).\] When $k=2$,
\begin{align*}
\det(P_2)&=(x-a_{11})(x-a_{22})-a_{12}a_{21} \\
&=x^2+\left(\frac{\mu(1)^2}{t}+\frac{\mu(2)^2}{t}-\frac{\mu(1)\mu(2)}{t}\right)x+\frac{\mu(1)^2\mu(2)^2}{t^2} \\
&=\mathcal{C}_2(x,t).
\end{align*}
Thus, $\det(P_k)$ and $\mathcal{C}_k(x,t)$ also have the same initial conditions. Therefore, \[\chi(x) = \det(P_{p-1}) = \mathcal{C}_{p-1}(x,t).\] 
\end{proof}

\cform*

\begin{proof}
Refer to \autoref{subsec:properties} for the definitions of $\alpha_k$, $\beta_k$, $A_k$, $B_k$, and $\gamma_k$. We use a strong induction argument to show that $\mathcal{C}_k$ has the above form. We check the base cases directly.
    \[\begin{alignedat}{1}
        \mathcal{C}_0(x,t) &= 1 = \frac{1}{t^0}(1) \\
        \mathcal{C}_1(x,t) &= \frac{\mu(1)^2}{t}+x =\begin{cases}
             \frac{1}{t}+x = \frac{1}{t^1}(x^0t^0+x^1t^1), &\ \varepsilon_1 = 1 \\
             t+x = \frac{1}{t^0}(x^1t^0 + x^0t^1), &\ \varepsilon_1 = -1             
             \end{cases}\\
        \mathcal{C}_2(x,t) &= \left(\frac{\mu(2)^2}{t}+x\right)\left(\frac{\mu(1)^2}{t}+x\right)-\frac{\mu(2)\mu(1)x}{t}   \\
        &= \frac{\mu(1)^2\mu(2)^2}{t^2} + \frac{\mu(1)^2x}{t} + \frac{\mu(2)^2x}{t} - \frac{\mu(2)\mu(1)x}{t} + x^2  \\
        &= \begin{cases}
             \frac{1}{t^2} + \frac{x}{t} + x^2 = \frac{1}{t^2}\left(x^0 + \big(1(x)-0+\frac{0}{x}\big)x^0t^1 + x^2t^2\right), &\ \varepsilon_1 = \varepsilon_2 = 1 \\
             1 + \frac{x}{t} + xt -x + x^2 = \frac{1}{t^1}\left(x^1 + \big(1(x)-1+\frac{1}{x}\big)x^1t^1 + x^1t^2\right), &\ \varepsilon_1 = -\varepsilon_2 \\
             t^2 + xt + x^2 = \frac{1}{t^2}\left(x^2 + \big(0(x)-0+\frac{1}{x}\big)x^2t^1 + x^0t^2\right), &\ \varepsilon_1 = \varepsilon_2 = -1 \\
             \end{cases}
    \end{alignedat}\]
    
Now assume that $\mathcal{C}_{\ell}$ has the hypothesized form for all $\ell < k$. We want to show that $\mathcal{C}_k$ also has this form.
Recall the recursive definition of $\mathcal{C}_k$:
\begin{align*}
\mathcal{C}_k(x,t) &= \left(\frac{\mu(k)^2}{t}+x\right)\mathcal{C}_{k-1}(x,t) - \left(\frac{\mu(k-1)\mu(k)x}{t}\right)\mathcal{C}_{k-2}(x,t) \\
&= \left(\frac{\mu(k)^2}{t}+x\right)\sum_{i = 0}^{k-1}\mathcal{C}_{k-1}^i(x)t^{i-\alpha_{k-1}} - \left(\frac{\mu(k-1)\mu(k)x}{t}\right)\sum_{i = 0}^{k-2}\mathcal{C}_{k-2}^i(x)t^{i-\alpha_{k-2}}. 
\end{align*}
We have four cases, summarized in \autoref{tab:cases}.

    \begin{table}[H]
        \centering
        \begin{tabular}{c|c|c|c|c}
            & $\varepsilon_{k-1}$ & $\varepsilon_k$ & $\alpha_k$ & $\beta_k$
            \\
            \hline
            \textit{Case I} & $1$ & $1$ & $\alpha_k = \alpha_{k-1} + 1 = \alpha_{k-2} + 2$ & $\beta_k = \beta_{k-1} = \beta_{k-2}$
            \\
            \textit{Case II} & $-1$ & $1$ & $\alpha_k = \alpha_{k-1}+1 = \alpha_{k-2} +1 $& $\beta_k = \beta_{k-1} = \beta_{k-2} + 1$
            \\            
            \textit{Case III} & $1$ & $-1$ & $\alpha_k = \alpha_{k-1} = \alpha_{k-2} + 1$ & $\beta_k = \beta_{k-1} + 1 = \beta_{k-2} + 1$
            \\            
            \textit{Case IV} & $-1$ & $-1$ & $\alpha_k = \alpha_{k-1} = \alpha_{k-2}$ & $\beta_k = \beta_{k-1} + 1 = \beta_{k-2} + 2$
        \end{tabular}
        \vspace{1em}
        
        \begin{tabular}{c|c|c|c|c|c}
            & $\varepsilon_{k-1}$ & $\varepsilon_k$ & $A_k$ & $B_k$ & $\gamma_k$ 
            \\
            \hline
            \textit{Case I} & 1 & 1 & $A_k = A_{k-1}$ & $B_k = B_{k-1}$ & $\gamma_k = \gamma_{k-1}$
            \\
            \textit{Case II} & $-1$ & 1 & $A_k = A_{k-1} + 1$ & $B_k = B_{k-1} $ & $\gamma_k = \gamma_{k-1}+1$
            \\
            \textit{Case III} & 1 & $-1$ & $A_k = A_{k-1} $ & $B_k = B_{k-1} + 1$ & $\gamma_k = \gamma_{k-1}+1$
            \\
            \textit{Case IV} & $-1$ & $-1$ & $A_k = A_{k-1}$ & $B_k = B_{k-1}$ & $\gamma_k = \gamma_{k-1}$ 
        \end{tabular}
    \caption{Summary of the four cases needed in the proof of \autoref{lem:cform}.
    \label{tab:cases}
    }
    \end{table}

\noindent \textit{Case I:} We have
\begin{align*}
\mathcal{C}_k(x,t) &= \left(\frac{1}{t}+x\right)\sum_{i = 0}^{k-1}\mathcal{C}_{k-1}^i(x)t^{i-(\alpha_{k}-1)} - \left(\frac{x}{t}\right)\sum_{i = 0}^{k-2}\mathcal{C}_{k-2}^i(x)t^{i-(\alpha_{k}-2)} \\
&= \sum_{i = 0}^{k-1}\mathcal{C}_{k-1}^i(x)t^{i-\alpha_{k}} + \sum_{i = 0}^{k-1}x\mathcal{C}_{k-1}^i(x)t^{i-\alpha_{k}+1} - \sum_{i = 0}^{k-2}x\mathcal{C}_{k-2}^i(x)t^{i-\alpha_{k}+1}.  
\end{align*}
The lowest power of $t$ is $-\alpha_k$ and its coefficient is 
\[\mathcal{C}_{k-1}^0(x) = x^{\beta_{k-1}} = x^{\beta_{k}}.\]
The highest power of $t$ is $k - \alpha_k = \beta_k$ and its coefficient is
\[x\mathcal{C}_{k-1}^{k-1}(x) = x^{\alpha_{k-1}+1} = x^{\alpha_{k}}.\] 
The second coefficient $\mathcal{C}_k^1$ is found by collecting the coefficients of $t^{1-\alpha_k}$. This gives us
\[\begin{alignedat}{1}
\mathcal{C}_k^1(x) &= \mathcal{C}_{k-1}^1(x) + x\mathcal{C}_{k-1}^0(x)-x\mathcal{C}_{k-2}^0(x) \\
&= \left(A_{k-1}x-\gamma_{k-1}+\frac{B_{k-1}}{x}\right)x^{\beta_{k-1}} + x^{\beta_{k-1}+1}-x^{\beta_{k-2}+1}  \\
&= \left(A_kx-\gamma_k+\frac{B_k}{x}\right)x^{\beta_k}.
\end{alignedat}\]

\noindent \textit{Case II:} We have
\begin{align*}
\mathcal{C}_k(x,t) &= \left(\frac{1}{t}+x\right)\sum_{i = 0}^{k-1}\mathcal{C}_{k-1}^i(x)t^{i-(\alpha_{k}-1)} - x\sum_{i = 0}^{k-2}\mathcal{C}_{k-2}^i(x)t^{i-(\alpha_{k}-1)} \\
&= \sum_{i = 0}^{k-1}\mathcal{C}_{k-1}^i(x)t^{i-\alpha_{k}} + \sum_{i = 0}^{k-1}x\mathcal{C}_{k-1}^i(x)t^{i-\alpha_{k}+1} - \sum_{i = 0}^{k-2}x\mathcal{C}_{k-2}^i(x)t^{i-\alpha_{k}+1}.  
\end{align*}
The lowest power of $t$ is $-\alpha_k$ and its coefficient is
\[\mathcal{C}_{k-1}^0(x) = x^{\beta_{k-1}} = x^{\beta_{k}}.\]
The highest power of $t$ is $k - \alpha_k = \beta_k$ and its coefficient is
\[x\mathcal{C}_{k-1}^{k-1}(x) = x^{\alpha_{k-1}+1} = x^{\alpha_{k}}.\] 
The second coefficient $\mathcal{C}_k^1$ is found by collecting the coefficients of $t^{1-\alpha_k}$. This gives us
\begin{align*}
\mathcal{C}_k^1(x) &= \mathcal{C}_{k-1}^1(x) + x\mathcal{C}_{k-1}^0(x)-x\mathcal{C}_{k-2}^0(x) \\
&= \left(A_{k-1}x-\gamma_{k-1}+\frac{B_{k-1}}{x}\right)x^{\beta_{k-1}} + x^{\beta_{k-1}+1}-x^{\beta_{k-2}+1}\\
&= \left((A_{k-1}+1)x-(\gamma_{k-1}+1)+\frac{B_{k-1}}{x}\right)x^{\beta_{k-1}} \\
&= \left(A_{k}x-\gamma_{k}+\frac{B_{k}}{x}\right)x^{\beta_{k}}.
\end{align*}

\noindent \textit{Case III:} We have
\begin{align*}
\mathcal{C}_k(x,t) &= \left(t+x\right)\sum_{i = 0}^{k-1}\mathcal{C}_{k-1}^i(x)t^{i-\alpha_{k}} - x\sum_{i = 0}^{k-2}\mathcal{C}_{k-2}^i(x)t^{i-(\alpha_{k}-1)} \\
&= \sum_{i = 0}^{k-1}\mathcal{C}_{k-1}^i(x)t^{i-\alpha_{k}+1} + \sum_{i = 0}^{k-1}x\mathcal{C}_{k-1}^i(x)t^{i-\alpha_{k}} - \sum_{i = 0}^{k-2}x\mathcal{C}_{k-2}^i(x)t^{i-\alpha_{k}+1}.  
\end{align*}
The lowest power of $t$ is $-\alpha_k$ and its coefficient is 
\[x\mathcal{C}_{k-1}^0(x) = x^{\beta_{k-1}+1} = x^{\beta_{k}}.\]
The highest power of $t$ is $k - \alpha_k = \beta_k$ and its coefficient is
\[\mathcal{C}_{k-1}^{k-1}(x) = x^{\alpha_{k-1}} = x^{\alpha_{k}}.\] 
The second coefficient $\mathcal{C}_k^1$ is found by collecting the coefficients of $t^{1-\alpha_k}$. This gives us
\begin{align*}
\mathcal{C}_k^1(x) &= \mathcal{C}_{k-1}^0(x) + x\mathcal{C}_{k-1}^1(x)-x\mathcal{C}_{k-2}^0(x) \\
&= x^{\beta_{k-1}} + \left(A_{k-1}x-\gamma_{k-1}+\frac{B_{k-1}}{x}\right)x^{\beta_{k-1}+1}-x^{\beta_{k-2}+1} \\
&= \left(A_{k-1}x-(\gamma_{k-1}+1)+\frac{B_{k-1}+1}{x}\right)x^{\beta_{k-1}+1}   \\
&= \left(A_{k}x-\gamma_{k}+\frac{B_{k}}{x}\right)x^{\beta_{k}}. 
\end{align*}

\noindent \textit{Case IV:} We have
\begin{align*}
\mathcal{C}_k(x,t) &= \left(t+x\right)\sum_{i = 0}^{k-1}\mathcal{C}_{k-1}^i(x)t^{i-\alpha_{k}} - tx\sum_{i = 0}^{k-2}\mathcal{C}_{k-2}^i(x)t^{i-\alpha_{k}} \\
&= \sum_{i = 0}^{k-1}\mathcal{C}_{k-1}^i(x)t^{i-\alpha_{k}+1} + \sum_{i = 0}^{k-1}x\mathcal{C}_{k-1}^i(x)t^{i-\alpha_{k}} - \sum_{i = 0}^{k-2}x\mathcal{C}_{k-2}^i(x)t^{i-\alpha_{k}+1}.  
\end{align*}
The lowest power of $t$ is $-\alpha_k$ and its coefficient is
\[x\mathcal{C}_{k-1}^0(x) = x^{\beta_{k-1}+1} = x^{\beta_{k}}.\]
The highest power of $t$ is $k - \alpha_k = \beta_k$ and its coefficient is
\[\mathcal{C}_{k-1}^{k-1}(x) = x^{\alpha_{k-1}} = x^{\alpha_{k}}.\] 
The second coefficient $\mathcal{C}_k^1$ is found by collecting the coefficients of $t^{1-\alpha_k}$. This gives us
\[\begin{alignedat}{1}
\mathcal{C}_k^1(x) &= \mathcal{C}_{k-1}^0(x) + x\mathcal{C}_{k-1}^1(x)-x\mathcal{C}_{k-2}^0(x) \\
&= x^{\beta_{k-1}} + \left(A_{k-1}x-\gamma_{k-1}+\frac{B_{k-1}}{x}\right)x^{\beta_{k-1}+1} -x^{\beta_{k-2}+1} \\
&= \left(A_kx-\gamma_k+\frac{B_k}{x}\right)x^{\beta_k}.
\end{alignedat}\]
\end{proof}

%-------------------------------------------------%
\section{Classification results and examples} \label{sec:classification}

In this section, we investigate classification questions by applying the results from \autoref{sec:alexander} with the aim of identifying which knots can be realized as spiral knots. 
A priori, it is not clear whether all periodic fibered knots are spiral; however, as discussed in \autoref{subsec:properties}, \autoref{cor:deg} and \autoref{cor:coeff} obstruct this.
(In particular, these results imply that the knot $8_{21}$, which is periodic and fibered, is not spiral.)
We will see evidence in \autoref{subsec:examples} that spiral knots appear to make up a rather small subfamily of the class of periodic knots. 

In \autoref{subsec:invariants}, we determine a formula for the genus of any spiral knot and state a condition on the determinant of a spiral knot with $p$ even. Then in \autoref{subsec:torus}, we discuss properties of torus knots which do not extend to the broader class of spiral knots. Finally in \autoref{subsec:examples}, we tabulate low crossing spiral knots and discuss some implications of our tabulations.

\subsection{Invariants of spiral knots} \label{subsec:invariants}
One immediate consequence of our formula for the Alexander polynomials of spiral knots from \autoref{Ther:Alex} is the following statement about the determinants of spiral knots. Recall that evaluating the Alexander polynomial at $t=-1$ recovers the determinant.

\begin{proposition}
If $p$ is even, then $q \mid \det(S(p,q,\varepsilon))$.
\end{proposition}

\begin{proof}
Suppose $p$ is even. We show that $\mathcal{C}_{p-1}(x,-1)$ has the form
\[\mathcal{C}_{p-1}(x,-1) = (x - 1)Q(x),\]
where $Q(x)\in\mathbb{Z}[x]$. We proceed by induction. Note that when $t=-1$, $\mu(k)=\varepsilon_k$ for all $k\in\{1,\dots,p-1\}$, and the Laurent polynomials $\mathcal{C}_k(x,t)$ become polynomials $\mathcal{C}_k(x,-1)\in\mathbb{Z}[x]$. When $p= 2$, $\mathcal{C}_1(x,-1) = x - 1$. Assume $\mathcal{C}_m(x,-1) = (x-1)q(x)$ for some odd $m$ and polynomial $q(x)\in\mathbb{Z}[x]$. Then, by substituting $t = -1$ into the recursive formula from \autoref{Ther:Alex}, we have
\begin{align*}
\mathcal{C}_{m+2}(x,-1) &= (x - 1)\mathcal{C}_{m+1}(x,-1) + \varepsilon_{m+1}\varepsilon_m x\mathcal{C}_{m}(x,-1) \\
&= (x - 1)\mathcal{C}_{m+1}(x,-1) + \varepsilon_{m+1}\varepsilon_m x(x-1)q(x) \\
&= (x-1)(\mathcal{C}_{m+1}(x,-1)+ \varepsilon_{m+1}\varepsilon_m xq(x)).
\end{align*}
Defining $Q(x) \vcentcolon =\mathcal{C}_{m+1}(x,-1)+ \varepsilon_{m+1}\varepsilon_m xq(x)\in\mathbb{Z}[x]$ yields the claimed form for $\mathcal{C}_{p-1}(x,-1)$ for all even $p$. 

Now we evaluate our formula for the Alexander polynomial at $t=-1$, assuming $p$ is even:
\begin{align*}
\Delta(-1) &= \prod_{\ell=1}^{q-1}\mathcal{C}_{p-1}(e^{\frac{2\pi \ell i}{q}},-1) \\
&= \prod_{\ell=1}^{q-1}(e^{\frac{2\pi \ell i}{q}}-1)Q(e^{\frac{2\pi \ell i}{q}}) \\
& = \prod_{\ell=1}^{q-1}(e^{\frac{2\pi \ell i}{q}}-1)\prod_{\ell=1}^{q-1}Q(e^{\frac{2\pi \ell i}{q}}) \\
&= (-1)^{q-1}q \prod_{\ell=1}^{q-1}Q(e^{\frac{2\pi \ell i}{q}}) \\ 
&= (-1)^{q-1}q \textit{Res}(1 + x + x^2 + \ldots + x^{q-1}, Q(x)),
\end{align*}
where $\textit{Res}(\cdot,\cdot)$ is the resultant of two polynomials. Since $Q(x)$ has integer coefficients, $\textit{Res}(\sum_{i=0}^{q-1}x^i, Q(x))$ is the determinant of an integer matrix, which is also an integer. Therefore $q \mid \Delta(-1) = \det(S(p,q,\varepsilon))$.
\end{proof}

We now present a formula for the genus of a spiral knot which follows quickly from our results on the Alexander polynomial. Recall that the degree of the Alexander polynomial $\Delta$ and genus $g$ of a knot $K$ are related by $\frac{\deg{\Delta}}{2} \leq g$. When $K$ is a spiral knot, \autoref{cor:deg} implies $g \geq \frac{(p-1)(q-1)}{2}$, and \autoref{lem:upgenus} implies $g \leq \frac{(p-1)(q-1)}{2}$. We therefore obtain a formula for the genus of any spiral knot.

\genus

Note that in particular, $S(p,q,\varepsilon)=S(p',q',\varepsilon')$ implies $(p-1)(q-1)=(p'-1)(q'-1)$. Additionally, the above genus formula agrees with the genus formula for torus knots, so we have found an invariant of torus knots which generalizes to the larger family of spiral knots.

Just as \autoref{Ther:Alex} generalizes to give a formula for the single-variable Alexander polynomials of spiral links, a modified version of \autoref{Ther:genus} also holds for spiral links. The minimum genus of a Seifert surface for the spiral link $S(n,k,\varepsilon)$ is \[g(S(n,k,\varepsilon)) = \frac{(n-1)(k-1)-(\gcd(n,k)-1)}{2}.\] See \autoref{sec:future} for more details.

Note that we may obtain \autoref{Ther:genus} independently using Edmonds' conditions \cite{Edmonds}, which in this context state that there are nonnegative integers $\Gamma$ and $\Lambda$ such that \[ g(S(p,q,\varepsilon))=q\Gamma+\frac{(q-1)(\Lambda-1)}{2},\] where $\Lambda\geq p$ (see also Chapter 8 of \cite{Livingston}). This independently gives the lower bound $\frac{(p-1)(q-1)}{2}$.

\subsection{Comparing spiral knots to torus knots} \label{subsec:torus}

One might wonder which additional properties of torus knots generalize to spiral knots. For instance, torus knots have the interesting property that switching the values of $p$ and $q$ does not change the knot; that is, $T_{p,q} = T_{q,p}$. In contrast with \autoref{Ther:genus}, among all spiral knots, this property holds \textit{only} for torus knots. We first show that a torus knot admits a unique spiral representation up to swapping $p$ and $q$.

\begin{proposition} \label{prop:torus}
$S(p,q,\varepsilon)$ and $S(q,p,\varepsilon)$ for $\varepsilon = (1,\ldots,1)$ (resp. $\varepsilon=(-1,\ldots,-1)$) are the unique spiral representations of the torus knot $T_{p,q}$ (resp. $T_{-p,q}=T_{p,-q}$) for $p,q>1$.
\end{proposition}

\begin{proof}
    First, observe that $S(p,q,(1,\dots,1))=T_{p,q}=T_{q,p}=S(q,p,(1,\dots,1))$, so we have the two spiral representations for $T_{p,q}$ as claimed.

    Toward proving the uniqueness of these representations, we first show that the second coefficient of the Alexander polynomial of a torus knot is always $\pm 1$. Let $T_{p,q}$ be a torus knot with $p,q>1$. Then we can represent $T_{p,q}$ as $S(p,q,\varepsilon)$ where $\varepsilon = (1,\ldots,1)$. Recall that $\gamma_{p-1}(\varepsilon)$ counts the sign changes of $\varepsilon$, which is $0$ in this case. We have by \autoref{cor:coeff} that the second coefficient of the Alexander polynomial of $T_{p,q} = S(p,q,\varepsilon)$ is $\pm(q\gamma_{p-1}(\varepsilon)+1) = \pm(q\cdot 0+1) = \pm 1$.

    Now suppose that $S(p',q',\varepsilon)$ for some $p',q',\varepsilon$ equals $T_{p,q}$ for $p,q>1$. We will show that we must have $\{p',q'\}=\{p,q\}$ and $\varepsilon=(1,\ldots,1)$. Because $S(p',q',\varepsilon)$ is a torus knot, we have by the above argument that the second coefficient of its Alexander polynomial is $\pm1$. Note that $q'\gamma_{p'-1}(\varepsilon)+1$ cannot equal $-1$ since $q'\gamma_{p'-1}(\varepsilon)$ is a product of two nonnegative integers. So $q'\gamma_{p'-1}(\varepsilon)+1=1$ and $\gamma_{p'-1}(\varepsilon)=0$, meaning $\varepsilon = (1,\ldots,1)$ or $(-1,\ldots,-1)$. Thus $S(p',q',\varepsilon)$ is a torus knot, and either $T_{p,q}=S(p',q',\varepsilon)=T_{p',q'}$ or $T_{p,q}=S(p',q',\varepsilon)=T_{-p',q'}$. 
    Therefore, because $p,q,p',q'>1$, we must have $S(p',q',\varepsilon)=T_{p',q'}$, since the isotopy class of the nontrivial torus knot $T_{r,s}$ determines the set $\{r,s\}$ up to an overall sign.
    The proof for $S(p',q',\varepsilon)=T_{-p,q}$ is entirely similar.
\end{proof}

Now we show that torus knots are the only spiral knots for which we can swap $p$ and $q$.

\swappq

\begin{proof}
Suppose $S(p,q,\varepsilon)$ is a torus knot. By \autoref{prop:torus},  $\varepsilon = (1,\dots,1)$ and $S(p,q,\varepsilon)=T_{p,q}$, or $\varepsilon = (-1,\dots,-1)$ and $S(p,q,\varepsilon)=T_{p,-q}$. In the first case, let $\varepsilon' = (1,\dots,1)$. Then $S(p,q,\varepsilon) = T_{p,q} = T_{q,p} = S(q,p,\varepsilon')$. The proof in the second case is the same.

Conversely, suppose $S(p,q,\varepsilon) = S(q,p,\varepsilon')$ for some $\varepsilon'$.
\autoref{cor:coeff} implies 
$(\Delta_{S(p,q,\varepsilon)})^1 = \pm(\gamma_{p-1}(\varepsilon)q + 1)$ and $(\Delta_{S(q,p,\varepsilon')})^1 = \pm(\gamma_{q-1}(\varepsilon')p + 1)$. So
$\gamma_{p-1}(\varepsilon)q = \gamma_{q-1}(\varepsilon')p.$
Since $\gcd(p,q) = 1$ and $p,q > 1$, $p$ must divide $\gamma_{p-1}(\varepsilon)$.
But $\gamma_{p-1}(\varepsilon)$ is the number of sign changes in $\varepsilon$, so $0\leq\gamma_{p-1}(\varepsilon) \leq p-2 < p$. It follows that $\gamma_{p-1}(\varepsilon) = 0$, which implies that $S(p,q,\varepsilon)$ is a torus knot. Note that by \autoref{prop:torus}, either $\varepsilon=(1,\ldots,1)$ and $\varepsilon'=(1,\ldots,1)$, or $\varepsilon=(-1,\ldots,-1)$ and $\varepsilon'=(-1,\ldots,-1)$.
\end{proof}

\begin{remark} \label{rmk:facts} We note some additional properties of torus knots that do not extend to spiral knots. Nontrivial torus knots are known to not be slice and have crossing number $\min((p-1)q, (q-1)p)$, but we can see from our tables in \autoref{subsec:examples} that in general the same is not true for spiral knots. This is a recovery of known results, as $10_{123}$ and $4_1$ were known to be weaving \cite{DiPrisaSavk}, but we collect them here for completeness. Finally, the Alexander polynomial classifies torus knots up to mirror image, but it turns out that the same is not true for spiral knots.
    \begin{enumerate}
        \item[(1)] There exist slice spiral knots. $S(3,5,(1,-1)) = 10_{123}$ is a slice spiral knot. As such, the Milnor conjecture does not extend to spiral knots. 
        \item[(2)] Not all spiral knots have crossing number $\min((p-1)q, (q-1)p)$. $S(3,2,(1,-1)) = 4_1$ is a spiral knot with crossing number $4 > 3 = \min((2)2, (1)3)$.
        \item[(3)] The Alexander polynomial does not classify spiral knots. The spiral knots in \autoref{fig:difpair} have the same Alexander polynomial but different Jones polynomials.
        \item[(4)] Furthermore, the Alexander polynomial and Jones polynomial together are still not enough to classify spiral knots. The spiral knots in \autoref{fig:samepair} have the same Alexander polynomial and Jones polynomial. However, they can be distinguished by their second cabled Jones polynomial, which we computed using KnotFolio \cite{KnotFolio}. 
    \end{enumerate}
Additionally, the pairs of spiral knots $S(9,q,\varepsilon_1)$ and $S(9,q,\varepsilon_2)$ for $\varepsilon_1$ and $\varepsilon_2$ as in \autoref{fig:samepair} continue to have the same Alexander polynomial for $q$ up to (at least) 13, although their Jones polynomials are only the same for $q=2$.
\end{remark}

\begin{figure}[ht!]
\centering
    \begin{minipage}[c]{0.45\linewidth}   
    \includegraphics[width=\linewidth]{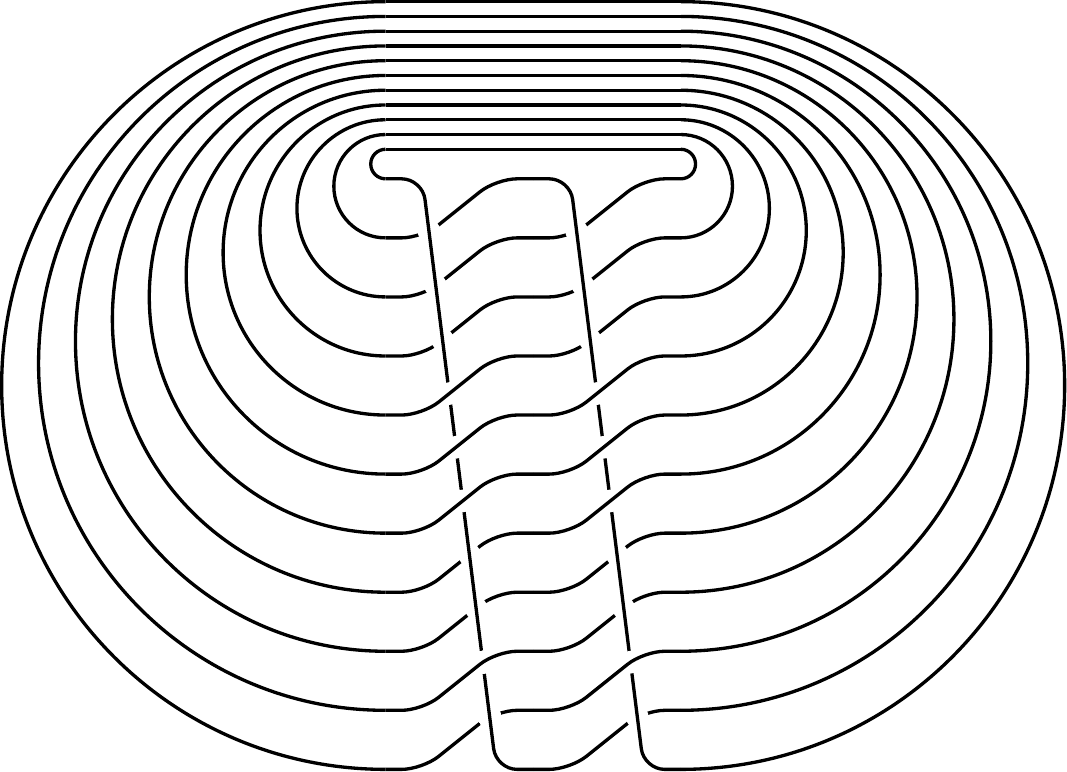}
    \caption*{S(11,2,(1,1,1,-1,-1,-1,1,1,-1,1))}
    \end{minipage}
    \hspace{2em}
    \begin{minipage}[c]{0.45\linewidth}
    \includegraphics[width=\linewidth]{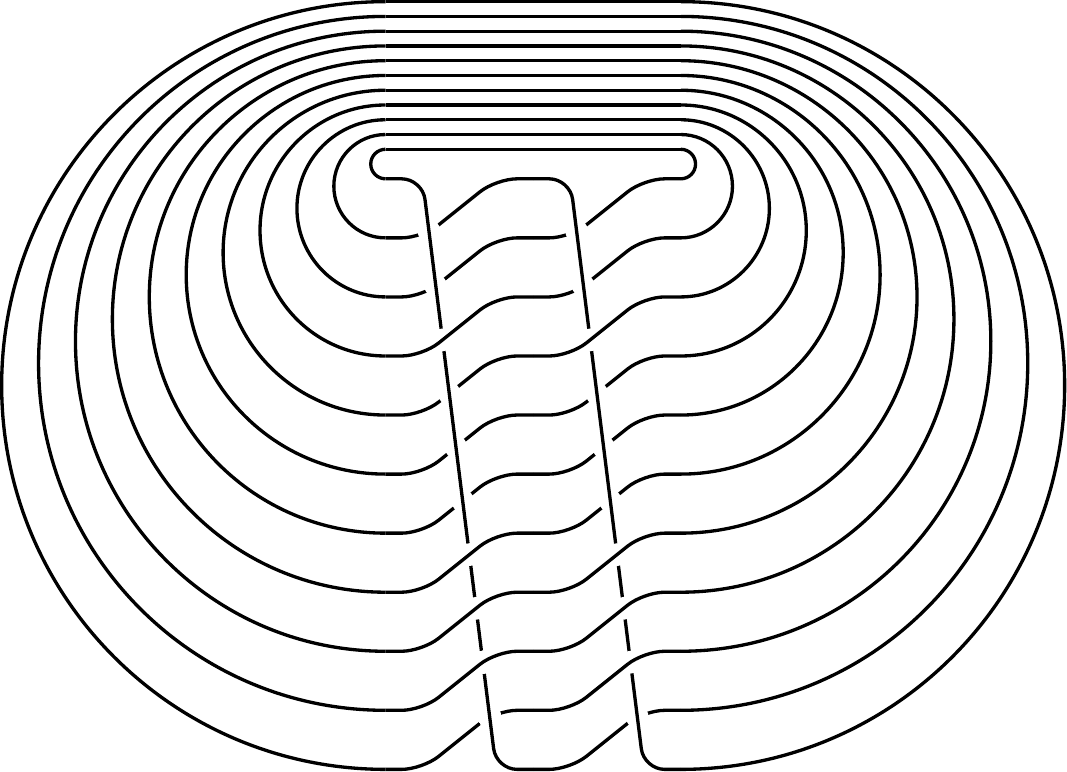}
    \caption*{S(11,2,(1,1,-1,1,1,1,-1,-1,-1,1))}
    \end{minipage}
    \caption{Two spiral knots with the same Alexander polynomial and different Jones polynomials, computed with SageMath via CoCalc \cite{sagemath}. \label{fig:difpair}}
\end{figure}

\begin{figure}[ht!]
\centering
    \begin{minipage}[c]{0.4\linewidth}   
    \includegraphics[width=\linewidth]{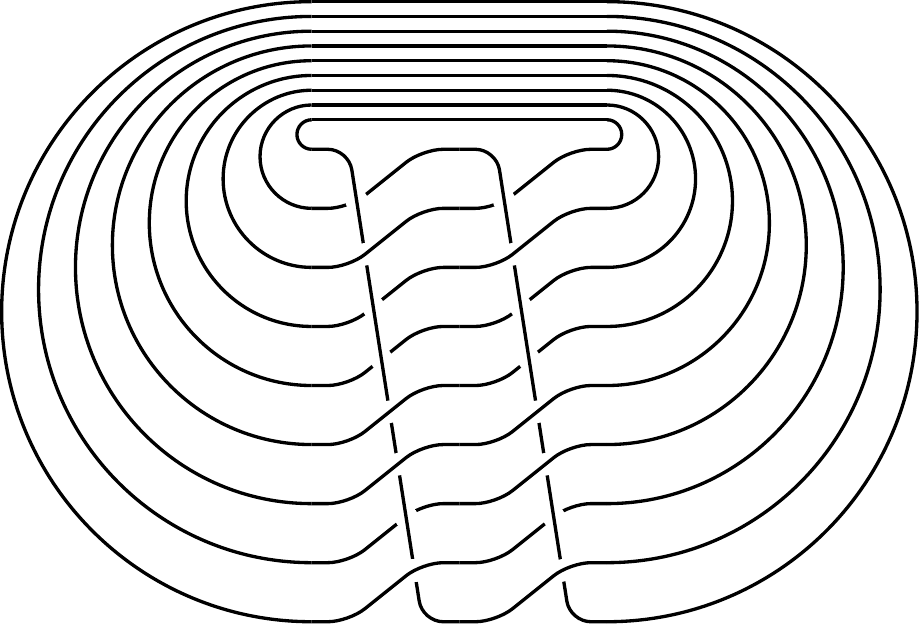}
    \caption*{S(9,2,(1,-1,1,1,-1,-1,1,-1))}
    \end{minipage}
    \hspace{2em}
    \begin{minipage}[c]{0.4\linewidth}
    \includegraphics[width=\linewidth]{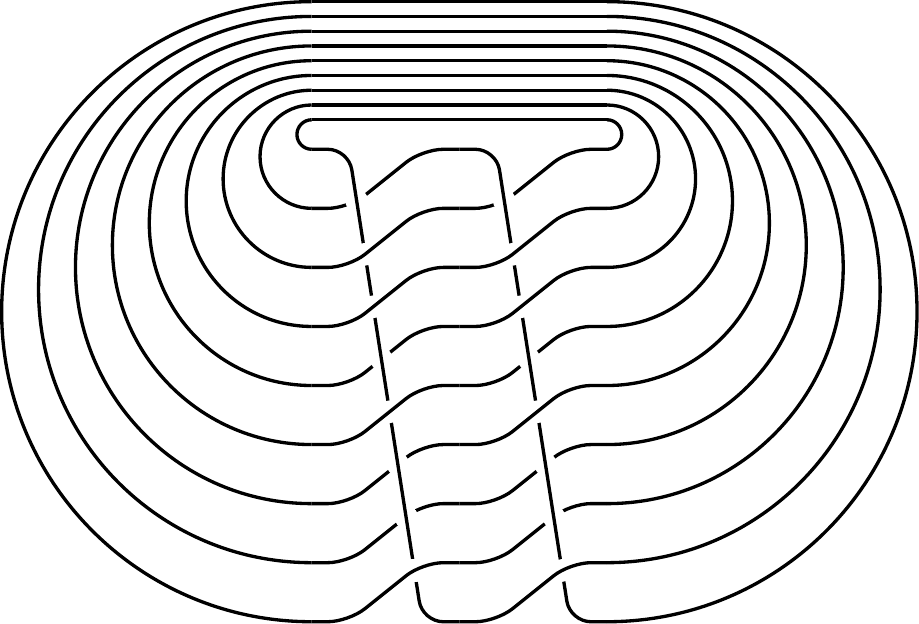}
    \caption*{S(9,2,(1,-1,-1,1,-1,1,1,-1))}
    \end{minipage}
    \caption{Two (distinct) spiral knots with the same Alexander polynomial and Jones polynomial, computed with SageMath via CoCalc \cite{sagemath}. They are distinguished by their second cabled Jones polynomial, computed via KnotFolio \cite{KnotFolio}. \label{fig:samepair}}
\end{figure}

\subsection{Classification of small spiral knots} \label{subsec:examples}

Now we present a tabulation of spiral knots with low crossing number up to mirror image. Since $2g$ is a lower bound for the crossing number of a knot, the crossing number $c$ of a spiral knot $S(p,q,\varepsilon)$ is bounded below by $(p-1)(q-1) \le c$. If $S(p,q,\varepsilon)$ has crossing number 12 or fewer, $p$ and $q$ must satisfy $(p-1)(q-1) \leq 12$. To identify every prime spiral knot with crossing number less than or equal to 12, we considered every spiral knot satisfying this upper bound, and checked with SnapPy \cite{SnapPy} that these knots are all prime. For every such spiral knot, we either identified the knot among those in the knot table with $c\leq 12$ using SnapPy \cite{SnapPy}, or obstructed this by showing no such knot in the knot table had matching Alexander polynomial. 
We summarize our findings in \autoref{tab:table1} and \autoref{tab:table2}, where an entry
\begin{itemize}
\item is left blank when the corresponding spiral knot does not satisfy $(p-1)(q-1) \leq 12$, 
\item contains a dash if $p$ and $q$ are not co-prime, 
\item contains an X if the knot was not among those in the knot table with crossing number $c\leq 12$, 
\item contains the name of the knot if it was present in the knot table. 
\end{itemize}
For the sake of simplicity, in our tables we adapt the convention that $\varepsilon_1 = 1$. (Negating $\varepsilon$ of a spiral knot will produce its mirror image.) We also only consider $\varepsilon$ vectors up to inverting the order of the coordinates. Note that no spiral knot with $p=6$ has crossing number low enough to appear in our tables, given the restrictions on $q$, which is why $p$ jumps from $5$ to $7$ between \autoref{tab:table1} and \autoref{tab:table2}.

\begin{table}[h]
    \centering
    \begin{tabular}{c|c|c|c|c|c|c}
         $\varepsilon \backslash$q & 2 & 3 & 4 & 5 & 6 & 7 \\
         \hline
         (1,1) & $3_1$ & -- & $8_{19}$ & $10_{124}$ & -- & X \\
         \hline
         (1,-1) & $4_1$ & -- & $8_{18}$ & $10_{123}$ & -- & X \\
         \hline
         (1,1,1) & -- & $8_{19}$ & -- & X  & -- &\\
         \hline
         (1,1,-1) & -- & $9_{47}$ & -- & X & -- &\\
         \hline
         (1,-1,1) & -- & $9_{40}$ & -- & X & -- &\\
         \hline
         (1,1,1,1) & $5_1$ & $10_{124}$ & & -- & &\\
         \hline
         (1,1,1,-1) & $6_2$ & $11n_{133}$ & & -- & &\\
         \hline
         (1,1,-1,1) & $7_6$ & $12n_{839}$& & -- & &\\
         \hline
         (1,1,-1,-1) & $6_3$ & $12n_{706}$& & -- & &\\
         \hline
         (1,-1,1,-1) & $8_{12}$ & $12a_{1019}$ & & -- & &\\
         \hline
         (1,-1,-1,1) & $7_7$ & $12n_{837}$& & -- & &\\
    \end{tabular}
    \caption{Spiral knots with low crossing number, tabulated by $\varepsilon$ and $q$.}
    \label{tab:table1}
\end{table}

\begin{table}[h]
    \centering
    \begin{tabular}{c|c|c|c|c}
(1,1,1,1,1,1) & (1,1,1,1,1,-1) & (1,1,1,1,-1,1) & (1,1,1,1,-1,-1) & (1,1,1,-1,1,1) \\
\hline
$7_1$ & $8_2$ & $9_{11}$ & $8_7$ & $9_{20}$ \\
\hline
(1,1,1,-1,1,-1) & (1,1,1,-1,-1,1) & (1,1,1,-1,-1,-1) & (1,1,-1,1,1,-1) & (1,1,-1,1,-1,1) \\
\hline
$10_{29}$ & $9_{26}$ & $8_9$ & $10_{44}$ & $11a_{159}$ \\
\hline
(1,1,-1,1,-1,-1) & (1,1,-1,-1,1,1) & (1,1,-1,-1,1,-1) & (1,1,-1,-1,-1,1) & (1,-1,1,1,1,-1) \\
\hline
$10_{43}$ & $9_{31}$ & $10_{42}$ & $9_{27}$ & $10_{41}$ \\
\hline
(1,-1,1,1,-1,1) & (1,-1,1,-1,1,-1) & (1,-1,1,-1,-1,1) & (1,-1,-1,1,1,-1) & (1,-1,-1,-1,-1,1) \\
\hline
$11a_{121}$ & $12a_{477}$ & $11a_{96}$ & $10_{45}$ & $9_{17}$ \\
    \end{tabular}
    \caption{Every spiral knot of the form $S(7,2,\varepsilon)$, tabulated by $\varepsilon$.}
    \label{tab:table2}
\end{table}

We can make a number of observations by referencing \autoref{tab:table1} and \autoref{tab:table2}. Right away we see that spiral knots appear to be a relatively small family of knots; there are 39 spiral knots
with minimum crossing number less than or equal to 12, compared with the 676 periodic fibered prime knots and 2,978 total prime knots with less than or equal to 12 crossings!
% 1,732 periodic knots and 2,978 prime knots with less than or equal to 12 crossings!
(In these counts we are not considering mirror images to be distinct.)
% Found # periodic knots by going to KnotInfo, doing advance search with symmetry group equal to 1 and crossing number less than 13, subtracting this # from 2,978

Further, one remarkable consequence of Thurston's work concerning the geometry of $3$-manifolds is that every knot is contained in exactly one of the following families: (1) torus knots, (2) satellite knots, or (3) hyperbolic knots \cite{Thurston}. All torus knots are spiral knots, and we have produced many examples of hyperbolic spiral knots, which raises the question of whether all (periodic fibered) hyperbolic knots can be realized as spiral knots. However, as noted in \autoref{subsec:properties}, the knot $8_{21}$ is periodic, fibered, and hyperbolic, but not spiral.

\begin{corollary} \label{cor:hyp}
Not all (periodic fibered) hyperbolic knots are spiral knots.
\end{corollary}

%-------------------------------------------------%
\section{Discussion and future directions} \label{sec:future}

Our investigation into the classification of spiral knots has provided a fruitful line of research, although many interesting questions remain. In this section, we present several future directions of study, posing questions and presenting some conjectures along the way. 

\subsection*{Classification questions}

From our tables and \autoref{cor:hyp}, we know that we cannot realize all periodic fibered hyperbolic knots as spiral knots, but what else can we say about the relationship between spiral knots and Thurston's classification? As we produced \autoref{tab:table1} and \autoref{tab:table2} ``by hand,'' only knots with fairly low crossing number are included, so in particular we do not have an example of a spiral satellite knot.

\begin{question}
    Are all spiral knots either torus knots or hyperbolic knots? That is, do there exist spiral knots which are satellite knots?
\end{question}

Another direction of study asks when two spiral knots with the same $p$ and $q$ values are the same. Given some $p,q,\varepsilon$, and $\varepsilon'$, note that if $\varepsilon'$ is $\varepsilon$ with its coordinates inverted, then $S(p,q,\varepsilon')$ will be the same knot as $S(p,q,\varepsilon)$, and if $\varepsilon'$ is $\varepsilon$ with its coordinates negated, then $S(p,q,\varepsilon')$ will be the mirror image of $S(p,q,\varepsilon)$. This raises the following question: Are there any other conditions on $\varepsilon$ besides inversion or negation that ensure $S(p,q,\varepsilon)$ and $S(p,q,\varepsilon')$ are the same knot, up to mirror image?

\begin{question} \label{Q:pqpairs}
When is $S(p,q,\varepsilon)$ the same knot as $S(p,q,\varepsilon')$?
\end{question}

\noindent We consider the more general question of when $S(p,q,\varepsilon)=S(p',q',\varepsilon')$ for distinct triples $(p,q,\varepsilon)$ and $(p',q',\varepsilon')$ later in this section.

\subsection*{Extension to links}

Unlike knots, links in $S^3$ with multiple components have several associated free abelian covers and therefore several associated Alexander polynomials. Two of these are the \emph{multivariable Alexander polynomial} corresponding to the maximal abelian cover of the link exterior and the \emph{single-variable Alexander polynomial} corresponding to the infinite cyclic cover induced by the homomorphism from the link group to $\mathbb{Z}$ which sends each meridian to $1\in\mathbb{Z}$. The multivariable Alexander polynomial specializes to the single-variable Alexander polynomial by setting all variables equal to each other.

We note that it is fairly straightforward to extend \autoref{Ther:Alex} to the single-variable Alexander polynomial of a spiral link; that is, the same formula also applies to spiral links. Given a link $L\subset S^3$, the single-variable Alexander polynomial is obtained by computing $\Delta_L(t)=\det(M-tM^T)$, where $M$ is a Seifert matrix corresponding to a (possibly disconnected) Seifert surface for $L$. A spiral link $S(n,k,\varepsilon)$, where $n$ and $k$ are no longer assumed to be relatively prime, still admits a connected cake surface with a cake homology basis consisting of $(n-1)(k-1)$ elements (see \autoref{sec:alexander}). The form of the Seifert matrix, and therefore the computation of $\Delta_{S(n,k,\varepsilon)}(t)$, does not differ from the knot case seen in \autoref{sec:alexander}. Thus, \autoref{Ther:Alex} also holds for spiral links.

\autoref{Ther:genus} also extends to the link case. 
We still obtain an upper bound on the 3-genus \[g(S(n,k,\varepsilon))\leq \frac{(n-1)(k-1)-(\gcd(n,k)-1)}{2}\] from the cake surface.
To obtain the reverse inequality, we employ the fact that for any link $L$ with $|L|$ components, we have 
\[\deg(\Delta_L(t))\leq 2g(L)+(|L|-1).\]

One further avenue of study is to explore classification questions for spiral links, and whether the formula for the Alexander polynomial can yield other interesting consequences for links. We also pose the following question.

\begin{question}
    Is every component in a spiral link itself a spiral knot?
\end{question}

\subsection*{Murasugi formula}
Suppose $K\subset S^3$ is a knot of period $q$, that is, $K$ is invariant under some orientation-preserving diffeomorphism $T \colon S^3\to S^3$, $T\neq\id_{S^3}$, such that $q=\min\{k\in\mathbb{N}\,|\,T^k=\id_{S^3}\}$.
By the resolved Smith Conjecture, the fixed point set of $T$ is an unknot $U$ \cite{BassMorgan}.
Consider the $q$-fold branched cover $\pi \colon S^3\to S^3/T\cong S^3$, and let $\overline{K}=\pi(K)$ and $\overline{U}=\pi(U)$.
Note that $\overline{U}$ is also unknotted.
Let $L$ be the 2-component link $L=\overline{U}\sqcup\overline{K}$.
Murasugi proved that the Alexander polynomial $\Delta_K(t)$ of $K$ and the two-variable Alexander polynomial $\Delta_L(t_1,t_2)$ of $L$ are related by \[\delta_\lambda(t)\Delta_K(t)=\prod_{l=1}^q\Delta_L(e^{\frac{2\pi li}{q}},t),\] where $\lambda$ is the linking number of $\overline{U}$ and $\overline{K}$ and $\delta_\lambda(t)=(1-t^\lambda)/(1-t)$ \cite{Murasugi}.
This formula is explored further in \cite{DavisLivingston}, where Davis-Livingston recast it into the \emph{Murasugi conditions} on pairs $(\Delta(t),q)$, where $\Delta(t)$ is the Alexander polynomial of some knot and $q\in\mathbb{N}$.
They study which pairs satisfy the Murasugi conditions and, supposing the Murasugi conditions are satisfied, when such a pair is realized by a $q$-periodic knot.

We do not prove it here, but we expect that the Alexander polynomial formula from \autoref{Ther:Alex} is an explicit version of Murasugi's formula for spiral knots.
In the setting of spiral knots, $U$ is the braid axis for $S(p,q,\varepsilon)$, and $\overline{U}$ is the braid axis for the knot $\overline{K}=S(p,1,\varepsilon)$. It would be interesting to recontextualize the results in this paper in light of Murasugi's work.

\begin{problem}
    Study the connection between our work and Murasugi's work.
    When $K$ is a spiral knot, can Murasugi's formula be turned into the formula from \autoref{Ther:Alex}? 
    That is, are the factors which appear in each formula for the Alexander polynomial closely related or the same?
\end{problem}

\subsection*{Bridge number and meridional rank}

Our results in \autoref{Ther:genus} and \autoref{Ther:swappq} naturally led us to consider how we could strengthen the conditions under which two spiral knots are equal. Inspired by the notions of bridge number and meridional rank, as well as some of our observations in \autoref{rmk:facts}, we propose the following conjecture (see also \autoref{Q:pqpairs}).

\begin{conjecture} \label{conj:equal}
     If $S(p,q,\varepsilon) = S(p',q',\varepsilon')$, then either $p = p'$ and $q = q'$ or $S(p,q,\varepsilon)$ is a torus knot (so that $p = q'$ and $q = p'$).
\end{conjecture}

This conjecture is connected to Cappell and Shaneson's \textit{meridional rank conjecture} (see Problem 1.11 in Kirby's list \cite{KirbyList}), which posits that the bridge number and meridional rank of any link are equal. Given a knot $K$, a \textit{bridge} of a diagram $D$ for $K$ is an arc in $D$ that crosses over another arc. The \textit{bridge number} $b(K)$ is the minimum number of bridges over all diagrams of a knot. The \textit{meridional rank} of a knot is the minimum number of meridians required to generate the knot group. The meridional rank conjecture has been proven for several families of links, including torus links \cite{MeridionalRankTorusLinks}, but it is not yet known whether the conjecture holds for spiral knots. 

\begin{question} \label{Q:MRC}
    Does the meridional rank conjecture hold for spiral knots?
\end{question}

On the other hand, every knot $K$ in \autoref{tab:table1} and \autoref{tab:table2}, along with some specific families of spiral knots with higher crossing number which we have studied (although we omit the details here), has bridge number $b(K)$ equal to $\min(p,q)$. Since this is also true for torus knots \cite{Schubert,Schultens}, we conjecture that this holds for all spiral knots. 

\begin{conjecture} \label{conj:bridge}
   All spiral knots $S(p,q,\varepsilon)$ satisfy $b(S(p,q,\varepsilon)) = \min(p,q)$. 
\end{conjecture}

If \autoref{conj:bridge} was true, then \autoref{conj:equal} would follow directly, as we can only have $\frac{(p-1)(q-1)}{2} = \frac{(p'-1)(q'-1)}{2}$ and $\min\{p,q\} = \min\{p',q'\}$ if we have the equality of sets $\{p,q\} = \{p',q'\}$. If $p = q'$ and $q=p'$, then we have $S(p,q,\varepsilon) = S(q,p,\varepsilon')$ which implies that $S(p,q,\varepsilon)$ is a torus knot by \autoref{Ther:swappq}. Otherwise, we have $p=p'$ and $q=q'$.

One direction of \autoref{conj:bridge} is fairly straightforward: using the algorithm for finding upper bounds on bridge number described in \cite{Blair_etal}, one can show that $\min(p,q)$ is an upper bound for spiral knots. We omit the full details, but we provide a proof sketch below.

The other direction of \autoref{conj:bridge} is more difficult, and a possible future direction of study. It is well known that meridional rank bounds the bridge number from below. If one could prove that the meridional rank of spiral knots is bounded below by $\min(p,q)$, this would prove \autoref{conj:bridge} and thus \autoref{conj:equal}, as well as answer \autoref{Q:MRC} in the affirmative.

\begin{proposition}
       Given a spiral knot $S(p,q,\varepsilon)$, 
    \[b(S(p,q,\varepsilon)) \leq \min(p,q).\] 
\end{proposition}

\begin{proof}[Proof sketch]
We show that $b(S(p,q,\varepsilon))$ is bounded above by both $p$ and $q$, and therefore can be at most $\min(p,q)$. First note that the bridge number $b(K)$ is bounded above by the braid index $\beta(K)$. So    \[b(S(p,q,\varepsilon)) \leq \beta(S(p,q,\varepsilon)) \leq p.\]

To show $b(S(p,q,\varepsilon))\leq q$, we use the algorithm for finding upper bounds on bridge number described in \cite{Blair_etal}. Represent $S(p,q,\varepsilon)$ in its standard braid form (see \autoref{sec:background}), and call crossings corresponding to $\sigma_i$ in the braid word \textit{$i$-crossings}. We define a partial coloring on the strands of $S(p,q,\varepsilon)$, starting with the $q$ strands that are understrands at all of the $1$-crossings. We call this initial set $A_1$. By the structure of the braid and the coloring rule from \cite{Blair_etal}, each uncolored strand at a $1$-crossing can be colored with the coloring move, since it is adjacent to or under a strand already included in $A_1$. This extends the partial coloring to all strands at $1$-crossings.

We then apply strong induction over $i$-crossings: assuming we have included all strands at crossings up to index $i-1$ in our partial coloring, we can show that the braid structure guarantees that any strand entering an $i$-crossing must have previously touched an $(i-1)$-crossing, and therefore is already colored. By the coloring move rule, there is a sequence of coloring moves that allows us to successively color all the remaining strands at each $i$-crossing. By induction, we can iterate this process until all the strands have been colored. Since the entire diagram is colored starting from $q$ initial strands, we get \[b(S(p,q,\varepsilon))\leq q\] as desired.
\end{proof}

%-------------------------------------------------%
\printbibliography[title={Bibliography}]

\vspace{2em}

\noindent
\begin{tabular}{lcl}

\makecell[tl]{
Sarah Blackwell \\ University of Virginia \\
Email: \texttt{\href{mailto:blackwell@virginia.edu}{blackwell@virginia.edu}}\\
URL: \url{https://seblackwell.com/} \\
$\quad$}
& \hspace{3em} &
\makecell[tl]{
Luke Moyar \\ University of Virginia \\
Email: \texttt{\href{mailto:pxq5zg@virginia.edu}{pxq5zg@virginia.edu}}
} \\
\makecell[tl]{
Ashish Das \\ North Carolina State University\\
Email: \texttt{\href{mailto:akdas3@ncsu.edu}{akdas3@ncsu.edu}}\\
$\quad$}
& &
\makecell[tl]{
Faisal Leo Quraishi \\ University of Nevada, Reno\\
Email: \texttt{\href{mailto:fquraishi@unr.edu}{fquraishi@unr.edu}}
} \\
\makecell[tl]{
Sydney Mayer \\ Vanderbilt University\\
Email:
\texttt{\href{mailto:sydney.mayer@vanderbilt.edu}{sydney.mayer@vanderbilt.edu}}
}
& &
\makecell[tl]{
Ryan Stees \\ University of Virginia\\
Email: \texttt{\href{mailto:rs2sf@virginia.edu}{rs2sf@virginia.edu}}\\
URL: \url{https://www.ryanstees.com}
}
\end{tabular}

\end{document}